\documentclass[a4paper,reqno,twoside]{amsart}
\usepackage{amssymb}
\usepackage{latexsym}
\usepackage{amsmath}
\usepackage{mathrsfs}
\usepackage[all,2cell]{xy}
        \UseAllTwocells \SilentMatrices
\usepackage{ifthen}

\theoremstyle{plain}
\newtheorem{propn}{Proposition}[section]
\newtheorem{thm}[propn]{Theorem}
\newtheorem{lemma}[propn]{Lemma}
\newtheorem{cor}[propn]{Corollary}

\theoremstyle{definition}
\newtheorem*{defn}{Definition}
\newtheorem*{caution}{Caution}

\newtheorem*{example}{Example}
\newtheorem*{examples}{Examples}

\theoremstyle{remark}
\newtheorem*{rem}{Remark}
\newtheorem*{rems}{Remarks}
\newtheorem*{notn}{Notation}

\newcommand{\ve}{\varepsilon}

\newcommand{\noisetwo}{\noise^{\oplus 2}}
\newcommand{\Kiltwo}{\Kil^{\oplus 2}}
\newcommand{\colT}{T^{[]}}
\newcommand{\colTdagger}{T^{\dagger[]}}
\newcommand{\colTconjugate}{T^{[]\dagger \Transpose}}
\newcommand{\colF}{F^{[]}}
\newcommand{\colFdagger}{F^{\dagger[]}}
\newcommand{\colFcdot}{F_\cdot^{[]}}
\newcommand{\colFdaggercdot}{F_\cdot^{\dagger[]}}
\newcommand{\colFt}{F_t^{[]}}

\newcommand{\QSSigma}{\Lambda^{\Sigma}}
\newcommand{\IntegSigma}{\mathbb{I}_\Sigma}
\newcommand{\IntegSigmadagger}{\mathbb{I}_\Sigma^\ddagger}
\newcommand{\ProcSigma}{\Proc_\Sigma}
\newcommand{\MartSigma}{\mathbb{M}_\Sigma}
\newcommand{\ProcSigmadagger}{\Proc_\Sigma^\ddagger}
\newcommand{\MartSigmadagger}{\mathbb{M}_\Sigma^\ddagger}
\newcommand{\kAu}{(\noise,\Init,\cyclic)}
\newcommand{\Init}{\mathsf{A}}
\newcommand{\InitNoise}{\mathsf{M}}
\newcommand{\cyclic}{\upsilon}
\newcommand{\ItoSigma}{\Ito^{\Sigma}}
\newcommand{\analytic}{\mathcal{A}}
\newcommand{\core}{\mathcal{C}}
\newcommand{\subspace}{\mathfrak{X}}
\newcommand{\dense}{\mathrm{o}}
\newcommand{\Sigmao}{\Sigma^{\dense}}
\newcommand{\Sigmaoprime}{\Sigma^{\dense}{}'}
\newcommand{\Weyl}{\mathcal{W}}
\newcommand{\Noise}{\Nil}
\newcommand{\Opdagger}{\mathcal{O}^\ddagger}
\newcommand{\umatrix}{\mathcal{M}}
\newcommand{\ucol}{\mathcal{C}}
\newcommand{\urow}{\mathcal{R}}
\newcommand{\umatrixdagger}{\mathcal{M}^\ddagger}
\newcommand{\ucoldagger}{\mathcal{C}^\ddagger}
\newcommand{\urowdagger}{\mathcal{R}^\ddagger}
\newcommand{\affiliated}{\eta\,}

\newcommand{\Transpose}{\Til}

\newcommand{\transpose}{\mathsf{t}}

\newcommand{\Domloc}{\Dom_{\rm{loc}}}

\newcommand{\Ito}{\mathcal{I}}
\newcommand{\initFock}{\mathfrak{H}}
\newcommand{\initFocknoise}{\mathfrak{H}_\noise}

\newcommand{\noise}{\mathsf{k}}
\newcommand{\domain}{\mathcal{D}}
\newcommand{\vp}{\varpi}

\newcommand{\Mil}{\mathsf{M}}
\newcommand{\Nil}{\mathsf{N}}

\newcommand{\Al}{\mathsf{A}}

\newcommand{\Hil}{\mathsf{H}}
\newcommand{\hil}{\mathsf{h}}
\newcommand{\Kil}{\mathsf{K}}
\newcommand{\kil}{\mathsf{k}}

\newcommand{\init}{\mathfrak{h}}
\newcommand{\Til}{\mathsf{T}}

\newcommand{\Exps}{\mathcal{E}}
\newcommand{\Fock}{\mathcal{F}}

\newcommand{\Op}{\mathcal{O}}

\newcommand{\Skorohod}{\mathcal{S}}

\newcommand{\khat}{\wh{\kil}}

\newcommand{\Proc}{\mathbb{P}}

\newcommand{\Real}{\mathbb{R}}
\newcommand{\Rplus}{\Real_+}

\newcommand{\Comp}{\mathbb{C}}
\newcommand{\Nat}{\mathbb{N}}

\newcommand{\Torus}{\mathbb{T}}

\newcommand{\wh}{\widehat}
\newcommand{\wt}{\widetilde}
\newcommand{\ol}{\overline}
\newcommand{\ul}{\underline}
\newcommand{\ot}{\otimes}

\newcommand{\otol}{\,\ol{\otimes}\,}
\newcommand{\otul}{\,\ul{\otimes}\,}

\newcommand{\op}{\oplus}

\newcommand{\tu}{\textup}

\DeclareMathOperator{\Graph}{Graph}
\DeclareMathOperator{\Dom}{Dom}
\DeclareMathOperator{\Ran}{Ran}
\DeclareMathOperator{\Lin}{Lin}

\DeclareMathOperator{\re}{Re}
\DeclareMathOperator{\im}{Im}

 \DeclareMathOperator{\loc}{loc}

\newcommand{\ip}[2]{\langle #1, #2 \rangle}

\newcommand{\norm}[1]{\lVert #1 \rVert}

\newcommand{\bra}[1]{\langle #1 |}
\newcommand{\ket}[1]{| #1 \rangle}

\newenvironment{alist}
{

\begin{enumerate}}
{\end{enumerate}}

\newenvironment{rlist}
{

\begin{enumerate}}
{\end{enumerate}}

\numberwithin{equation}{section} 

\begin{document}

\title
[Quasifree martingales]
 {Quasifree martingales}

\author[Martin Lindsay]{J.\ Martin Lindsay}
\author[Oliver Margetts]{Oliver T. Margetts}
\address{Department of Mathematics and Statistics,
 Fylde College,
Lancaster University, Lancaster LA1 4YF, U.K.}
\email{j.m.lindsay@lancaster.ac.uk, o.margetts@lancaster.ac.uk}

\subjclass[2010]{Primary
 81S25, 46L53; Secondary 47L60}
  \keywords{Noncommutative probability, quantum stochastic, quantum martingale,
 CCR algebra, symplectic map, Bogoliubov transformation, quasifree state, squeezed state}

 \begin{abstract}
 A noncommutative Kunita-Watanabe-type
 representation theorem is established
 for the martingales of
 quasifree states of CCR algebras.
 To this end the basic theory of quasifree stochastic
 integrals is developed
 using the abstract It\^o integral in symmetric Fock space,
 whose interaction with the
 operators of Tomita-Takesaki theory we describe.
 Our results extend earlier quasifree
 martingale representation theorems in two
 ways: the states are no longer assumed to be gauge-invariant, and
 the multiplicity space may now be infinite-dimensional.
 The former involves
 systematic exploitation of Araki's Duality Theorem.
 The latter requires the development of a
 transpose on matrices of unbounded operators,
 defying the lack of complete boundedness of the transpose operation.
 \end{abstract}
\maketitle

\section*{Introduction}
\label{section: introduction}

 In this paper we consider martingales adapted to a filtration of
 von Neumann algebras determined by a quasifree state of the CCR
 algebra over an $L^2$-space of vector valued functions on the half-line.
 The main tools of our analysis are the abstract It\^o integral in Fock space
 whose interaction with the operators of Tomita-Takesaki theory
 enables us to develop
 the basic theory of quasifree quantum stochastic integrals,
 and Araki's Duality Theorem for generating Type III factors with
 a cyclic and separating vector from the Fock representation of a
 CCR algebra.
 A transpose operation on the relevant class of integrands
 also plays a crucial role.
 The main result is a noncommutative
 Kunita-Watanabe-type representation theorem for quasifree
 martingales.

 Our results extend previous work in two ways.
 First the multiplicity space of the noise may now be
 infinite dimensional,
 and secondly,
 the class of quasifree states is much wider than hitherto
 considered;
 it is subject only to natural constraints,
 in particular we go beyond guage-invariant states.
 The importance of the former generalisation is underlined
 by the fact that the
  stochastic flows arising in the dilation of Markov
 semigroups on operator algebras typically require
 infinite-dimensional multiplicity spaces.
 A consequence of the latter is that
 (without guage invariance)
 creation and annihilation integrals need no longer be
 mutually orthogonal at the Hilbert space level.
 As with~\cite{HuL}, and its
 fermionic counterpart~\cite{Lfermionic},
 the full filtration of the quasifree noise is used here,
 rather than that generated by a fixed linear combination of
 quasifree quantum stochastic integrators, as in~\cite{BSWtwo}
 (the connection between these is elucidated in~\cite{LiW}).

 Recent developments in the use of quantum probabilistic models
 (e.g.~\cite{AtJ},~\cite{BeltonCMP}) demonstrate the need for
 quasifree stochastic analysis.
 In a sister paper (\cite{LM}) we develop a stochastic calculus for the
 quasifree integrals defined here.

 Noncommutative martingale representation theorems have been
 established in a variety of other contexts. The original one was for
 the Clifford filtration, which is the fermionic analogue of the
 Wiener filtration of canonical Brownian motion (\cite{BSWone}).
 Its free analogue was obtained in~\cite{BianeSpeicher}.
 A representation theorem for martingales with respect to the
 operator filtration of (minimal variance) quantum Brownian motion
 as Hudson-Parthasarathy quantum stochastic integrals, was obtained
 in~\cite{HLP} for the classes of essentially Hilbert-Schmidt and
 unitary martingales, in~\cite{PS1} for so-called regular martingales,
 and in~\cite{PS2} for regular martingales with respect to
 infinite dimensional quantum noise
 (see also~\cite{Meyer} and, for a recent
 coordinate-free treatment not reliant on extra set-theoretic axioms,~\cite{LQSI}).
 These results lie at a deeper level of noncommutativity than the Clifford
 and free cases, which make essential use of the finite trace
 available in those contexts.
 So far they cover only
  a class of bounded
 (as opposed to $L^2$-) martingales,
 however they do extend very satisfactorily
 to an algebra of semimartingales whose martingales
 precisely comprise the Parthasarathy-Sinha class
 (\cite{Attal}).
 White noise extensions of the latter
 form of martingale representation have been obtained
 in which explicit expression is found for the `stochastic derivatives'
 (see~\cite{JiO} and references therein).

The plan of the paper is as follows.
 In
 Section~\ref{section: affiliated operators}
 an extension of the well-known vector-operator
 correspondence for operators affiliated to a von Neumann algebra
 with cyclic and separating vector is established.
 The transpose operation that we need for defining
 quasifree stochastic integrals is identified,
 and its properties described, in
 Section~\ref{section: transpose and conjugate}.
 Commutation relations between the abstract It\^o integral in Fock
 space and operators which respect the Fock space filtration
 are proved in
 Section~\ref{section: Ito commutation relations}.
 In
 Sections~\ref{section: CCR algebras and qf states}
 and~\ref{section: qf states for stochastic analysis},
 the general context for our stochastic calculus is set,
 through a detailed discussion of relevant sufficient conditions
 for Araki's Duality Theorem to apply.
  Natural assumptions for the stochastic setting then emerge
 and these are shown to imply the sufficient conditions.
  We also describe
 classes of examples of quasifree states
 for stochastic calculus
 which are covered by our general assumptions.
 Section~\ref{section: modified Ito}
 establishes the underlying vector process theory
 by means of a modified It\^o integral and its
 commutation relations with the relevant Tomita-Takesaki
 $S$ operator,
 using results of
 Section~\ref{section: Ito commutation relations}.
 In the last section,
   quasifree stochastic integrals are defined and are shown
  to yield all the martingales of the theory, moreover
  adjointability of a martingale is shown to correspond
  precisely to the adjointability of the quasifree integrand
  process.
 Various facts that we need about the behaviour of unbounded
 operators under composition, orthogonal sum and tensor product
 are gathered in an appendix.

\noindent
\emph{Notational conventions}.
 For any vector-valued function $f: \Rplus \to V$ and subinterval $I$ of $\Rplus$,
 $f_I$ denotes the function agreeing with $f$ on $I$ and taking the value $0$ outside $I$.
  All Hilbert spaces are complex,
 with inner products linear in the second argument,
 in sinc with the following natural and very convenient
  (Dirac-inspired) notations: for
  a vector $u$ in the Hilbert space $\hil$, we write
  $\ket{u} \in \ket{\hil} := B(\Comp;\hil)$ and
 $\bra{u} \in \bra{\hil}:= B(\hil;\Comp)$ for the respective operators
 $\lambda\mapsto \lambda u$
 and $v\mapsto \ip{u}{v}$.
 We abbreviate $\hil\oplus \hil$ to $\hil^{\oplus 2}$.
 The linear span of a set of vectors $S$ is denoted $\Lin S$.
 For subspaces $U_1$ and $U_2$ of Hilbert spaces $\hil_1$ and $\hil_2$
 we write
 $U_1 \otul U_2$ for
 $\Lin \{u_1 \ot u_2: u_1\in U_1, u_2\in U_2 \}$,
 the linear tensor product of $(U_1,U_2)$
 realised in the
 Hilbert space tensor product $\hil_1 \ot \hil_2$.
 Blanks replace zero entries in matrices.

  The following notation is used for the symmetric Fock space
  over a Hilbert space $\hil$:
 $\Gamma(\hil) = \bigoplus_{n\geq 0} \hil^{\vee n}$,
 where $\hil^{\vee 0} =  \Comp$ and,
 for $n\geq 1$, $\hil^{\vee n}$ denotes the $n$-fold symmetric
 tensor power of $\hil$.
 The (normalised) exponential vectors are given by
\[
 \vp(u) := \exp(-\norm{u}^2/2)\, \ve(u) \ \text{ where } \
 \ve(u):= \big( (n!)^{-1/2}u^{\ot n}\big)_{n\geq 0}
 \qquad
 (u\in \hil),
\]
 and the Fock vacuum vector
 $\Omega_\hil$, by $\vp(0) = \ve(0) \in \Gamma(\hil)$.
 For $S \subset \hil$, we set
 $\Exps(S) := \Lin\{\ve(v): v\in S\}$.
    For $u\in\hil$, the Fock-Weyl operator $W_0(u)$
 is the unitary obtained by continuous linear extension
 of the inner-product preserving prescription
\begin{equation*}
 \vp(v) \mapsto e^{-i \im \ip{u}{v}} \vp(u+v)
 \qquad
 (v\in \hil).
\end{equation*}
  We also use the gradient operator $\nabla$ on Fock space
  (which will be freely ampliated without change of notation).
  This is the unique closed operator from
  $\Gamma(\hil)$ to $\hil \ot \Gamma(\hil)$
  with core $\Exps : = \Exps(\hil)$
  satisfying
  \[
  \nabla \ve(v) = v \ot \ve(v)
  \qquad (v\in\hil).
  \]

\section{Affiliated operators and matrix-operator correspondence}
 \label{section: affiliated operators}

 The following notations will be used for classes of unbounded operators.
 For a subspace $\domain_1$ of the Hilbert space $\Hil_1$,
 write $\Op(\domain_1;\Hil_2)$ for the linear space of operators from
 $\Hil_1$ to $\Hil_2$ with domain $\domain_1$ and,
 for dense subspaces
 $\domain_1$  of $\Hil_1$ and $\domain_2$ of $\Hil_2$, set
 \[
 \Opdagger(\domain_1,\domain_2)
 :=
 \big\{
 T \in \Op(\domain_1;\Hil_2):
 \Dom T^* \supset \domain_2
 \big\}
 \text{ and }
 T^\dagger := (T^*)_{|\domain_2}.
 \]
 Clearly the dagger operation is a conjugate-linear isomorphism
 \begin{equation}
 \label{dagger}
 \dagger: \Op^\ddagger(\domain_1,\domain_2)
 \to
 \Op^\ddagger(\domain_2,\domain_1)
 \end{equation}
 satisfying $T^{\dagger\dagger}= T$.
 In case the Hilbert spaces are the same,
 we abbreviate $\Opdagger(\domain,\domain)$ to
 $\Opdagger(\domain)$.

\begin{rem}
 By the Closed Graph Theorem,
 $\Opdagger(\Hil_1,\domain_2) = B(\Hil_1;\Hil_2)$,
 for any dense subspace $\domain_2$
 of $\Hil_2$.
\end{rem}

 \emph{For this section we fix a von Neumann algebra} $(\Mil,\Hil)$.
 There will be supplementary Hilbert spaces $\hil$, $\hil_1$ and $\hil_2$
 appearing. The following definition extends standard terminology
 (for the case where $\hil_1 = \hil_2 = \Comp$).

\begin{defn}
 A possibly unbounded operator $T$, from
 $\hil_1 \ot \Hil$ to  $\hil_2 \ot \Hil$,
 \emph{is affiliated to} $\Mil$, written
 $T \affiliated B(\hil_1;\hil_2) \otol \Mil$, if
 for all unitaries $u$ in $\Mil'$,
 $(I_2 \ot u^*) T (I_1 \ot u) = T$,
 in particular $(I_1 \ot u) \Dom T = \Dom T$.
\end{defn}
 \begin{rem}
 If $T$ is closed and densely defined then
 $T \affiliated B(\hil_1;\hil_2) \otol \Mil$
 if and only if
 $P_G \in B(\hil_1 \op \hil_2) \otol \Mil$,
 where $G = \Graph T$.
 \end{rem}
 For a subspace $\domain$ of $\Hil$, set
\[
 \Op_\Mil(\hil_1 \otul \domain; \hil_2 \ot \Hil) :=
 \big\{
 T \in \Op(\hil_1 \otul\domain; \hil_2 \ot \Hil):
  T \affiliated B(\hil_1;\hil_2)\otol \Mil
 \big\},
\]
 and if $\domain$ is dense, also set
 \[
 \Opdagger_\Mil(\hil_1 \otul \domain, \hil_2 \otul \domain) :=
 \big\{
 T \in \Opdagger(\hil_1 \otul \domain, \hil_2 \otul \domain) :
 T \affiliated B(\hil_1; \hil_2) \otol \Mil
 \big\},
 \]
 and abbreviate
 $\Opdagger_\Mil(\hil \otul \domain, \hil \otul \domain)$
  to
 $\Opdagger_\Mil(\hil \otul \domain)$.
 It is easily seen that,
 when $\domain_1 = \hil_1 \otul \domain$ and
 $\domain_2 = \hil_2 \otul \domain$,
  the
 conjugate-linear isomorphism~\eqref{dagger}
 restricts to an isomorphism
 \[
 \Opdagger_\Mil(\hil_1 \otul \domain, \hil_2 \otul \domain)
 \to
 \Opdagger_\Mil(\hil_2 \otul \domain, \hil_1 \otul \domain).
 \]

 \emph{For the rest of the section suppose that $\Mil$ has
 a cyclic and separating vector} $\xi$,
 set $\Xi = \Mil' \xi$, and let
 $S_\xi$ be the associated Tomita-Takesaki operator
 (\cite{TakesakiTwo}, Chapter VI;
 \cite{StZ}, Chapter 10).
 Define operators $E_\xi := I_\Hil \ot \ket{\xi}$ and
 $E^\xi := (E_\xi)^* = I_\Hil \ot \bra{\xi}$
 where the Hilbert space $\Hil$ is determined by context.
 Note that
\[
 \Opdagger_\Mil(\hil_1 \otul \Xi, \hil_2 \otul \Xi)
=
 \big\{ T \in \Op_\Mil(\hil_1 \otul \Xi, \hil_2 \ot \Hil):
  \hil_2 \subset \Dom T^*E_\xi
 \big\}.
\]

 The following class of operators helps us manage adjoints of
 affiliated operators through bounded operators:
 \begin{multline*}
 B^\ddagger_{\Mil,\xi}(\hil_1;\hil_2 \ot \Hil):=
 \big\{
 B \in B(\hil_1;\hil_2 \ot \Hil):
 \\
  \exists_{B_\dagger \in B(\hil_2;\hil_1\ot\Hil)} \
  \forall_{x'\in\Mil'} \
 B^* \big( I_2 \ot x'\big) E_{\xi} =
 E^\xi \big( I_1 \ot x' \big) B_{\dagger}
 \big\}.
 \end{multline*}
 When such an operator exists it is unique.
 The map
 \begin{equation*}
 B^\ddagger_{\Mil,\xi}(\hil_1;\hil_2 \otul \Hil)
 \to
 B^\ddagger_{\Mil,\xi}(\hil_2;\hil_1 \otul \Hil),
 \quad
 B \mapsto B_\dagger
 \end{equation*}
 is manifestly a conjugate-linear isomorphism
 satisfying $B_{\dagger\dagger} = B$.
 Clearly, for $B \in B(\hil_1; \hil_2 \ot \Hil)$,
 to be in
 $B^\ddagger_{\Mil, \xi}(\hil_1; \hil_2 \ot \Hil)$
 is for there to be a
 $B_\dagger \in B(\hil_2; \hil_1 \ot \Hil)$
 satisfying
 \begin{equation}
 \label{characterised}
 \ip{c_1 \ot x' \xi}{B_\dagger c_2} =
 \ip{B c_1}{c_2 \ot x'^* \xi}
 \quad
 (c_1 \in \hil_1, c_2 \in \hil_2, x' \in \Mil').
 \end{equation}
 Moreover,
 for $A \in B(\hil_1;\hil_2)$ and $\eta \in \Dom S_\xi$,
 \[
 A \ot \ket{\eta} \in B^\ddagger_{\Mil,\xi}(\hil_1; \hil_2 \ot \Hil)
 \  \text{ and }\
 \big( A \ot \ket{\eta} \big)_\dagger = A^* \ot \ket{S_\xi \eta}.
 \]
 Note also that, when
 $T \in \Opdagger_\Mil(\hil_1 \otul \Xi, \hil_2 \otul \Xi)$,
 the operator $T E_{\xi}$ is everywhere defined and closed,
 and thus bounded.
 The `matrix-operator' correspondences contained in
 the straightforward proposition below
 play a significant role in the sequel.
\begin{propn}
 \label{when D =}
 The map
\[
 \Op_\Mil(\hil_1 \otul \Xi; \hil_2 \ot \Hil)
 \to \Op(\hil_1; \hil_2 \ot \Hil),
 \quad
 T \mapsto T E_\xi := T E_{\xi}
\]
 is a linear isomorphism
 with inverse given by
 $B \mapsto B^\xi$, where
 $B^\xi$ is the linearisation of the bilinear map
 \[
 (c_1, x' \xi) \mapsto (I_2 \ot x') B c_1,
 \]
 which restricts to an isomorphism
 \[
 \Opdagger_\Mil(\hil_1 \otul \Xi, \hil_2 \otul \Xi)
 \to
 B^\ddagger_{\Mil, \xi}(\hil_1;\hil_2 \otul \Hil),
 \]
 intertwining the operations $^\dagger$ and $_\dagger$\tu{:}
\[
 \big(T E_{\xi}\big)_\dagger = T^\dagger E_{\xi}
 \ \text{ and } \
 (B^\xi)^\dagger = (B_\dagger)^\xi.
\]
\end{propn}

\begin{rems}
 To illustrate on simple tensors, let
 \[
 A \in \Op(\hil_1;\hil_2), \
 B \in B(\hil_1;\hil_2), \
 R \in \Op_\Mil(\Xi; \Hil), \
 X \in \Op_\Mil(\Xi)\ \text{ and }\ Z \in \Opdagger_\Mil(\Xi).
 \]
 Then, setting $\zeta = Z \xi$,
 \begin{align*}
 &(A \otul R) E_{\xi} = A \otul \ket{R \xi}
 \ \text{ and } \
 (A \otul \ket{X \xi})^\xi = A \otul X,
 \text{ so }
 \\
 &\big((B \otul R) E_{\xi}\big)_\dagger =
 \big( B \ot \ket{R \xi} \big)_\dagger =
 B^* \ot \ket{S_\xi R \xi} =
 B^* \ot \ket{R^\dagger \xi} =
 (B \otul R)^\dagger E_{\xi},
 \text{ and }
 \\
 &\big( (B \ot \ket{\zeta})_\dagger \big)^\xi =
 \big( B^* \ot \ket{S_\xi \zeta} \big)^\xi =
 B^* \otul Z^\dagger \subset
 (B \ot Z)^* =
 \big((B \ot \ket{\zeta})^{\xi}\big)^*,
 \end{align*}
  thus
 $\big( (B \ot \ket{\zeta})_\dagger \big)^\xi =
 \big( (B \ot \ket{\zeta})^\xi \big)^\dagger$.

 When $\hil_1 = \hil_2 = \Comp$
 the above correspondences reduce to the well-known
 linear isomorphism
 \begin{equation}
 \label{eqn: op vector}
 \Op_\Mil(\Xi; \Hil) \to \Hil,
 \quad
 X \mapsto X \xi,
 \end{equation}
 and its restriction, the isomorphism
 \begin{equation}
 \label{eqn: opdagger dom S}
 \Opdagger_\Mil(\Xi) \to \Dom S_\xi,
 \end{equation}
 under which $S_\xi (X \xi) = X^\dagger \xi$
 (see, for example,~\cite{BrR}, Proposition 2.5.9).
 Specifically,
 $\Op_\Mil(\Comp; \Comp \ot \Hil) = \ket{\Hil}$ and
 \[
 B^\ddagger_{\Mil,\xi}(\Comp;\Comp \otul \Hil) =
 \big\{ \ket{\zeta}: \zeta \in \Dom S_\xi
 \big\}
 \text{ with }
 \ket{\zeta}_\dagger = \ket{S_\xi \zeta}.
 \]
 In the next section we shall see how this
 connection can be raised to the matrix level.
\end{rems}

 We end this section with another very useful elementary fact.
 \begin{lemma}
 \label{ultrastrong lemma}
 Let $\mathcal{V}$ be a subspace of $\Mil'$.
 \begin{alist}
 \item
 If $\mathcal{V}$ is dense in $\Mil'$ in the strong operator topology
 then $\mathcal{V}\xi$ is a common core for
 all operators in $\Opdagger_\Mil(\Xi)$.
 \item
 If $\mathcal{V}$ is dense in $\Mil'$ in the ultrastrong topology
 then $\hil_1 \otul \mathcal{V}\xi$ is a common core for all operators in
 $\Opdagger_\Mil(\hil_1 \otul\Xi,\hil_2\otul\Xi)$.
 \end{alist}
 \end{lemma}

 \section{Transpose and conjugate for matrices of unbounded operators}
 \label{section: transpose and conjugate}

 \emph{For this section we fix a von Neumann algebra $(\Mil,\Hil)$
 with cyclic and separating vector $\xi$}, let
 $S_\xi$ and $F_\xi$ denote the corresponding Tomita-Takesaki
 operators, and set $\Xi = \Mil' \xi$.
 Also Hilbert spaces
 $\noise$, and $\noise_i$ ($i=0, 1, ... $),
 will appear which are complexifications of real Hilbert spaces;
 we denote
 the action of their associated conjugations $k$, respectively $k_i$,
 by $c \mapsto \ol{c}$.
 We consider a transpose operation on a class of
 abstract matrix spaces over a space of unbounded operators
 affiliated to $\Mil$.
 We then detail
 its relation to the dagger
 (adjoint) operation and to $S_\xi$.
 This is needed to handle quasifree stochastic integrals for
 infinite dimensional noise; it also enables
 multiple quasifree integrals to be defined in~\cite{LM},
 where they are used for solving
 quasifree stochastic differential equations.

 For $B\in B(\noise_1;\noise_2)$, its \emph{conjugate operator}
 is defined by
 \[
 \ol{B} := k_2 B k_1 \in B(\noise_1;\noise_2), \
 c \mapsto \ol{B\ol{c}},
 \]
 and its \emph{transpose} by
 $B^\transpose := \ol{B}^* = \ol{B^*}$.
 The transpose maps $B(\noise_1;\noise_2)$ linearly and isometrically onto
 $B(\noise_2;\noise_1)$.
 Due to the lack of complete boundedness of the transpose,
 the map
 \[
 B(\noise_1;\noise_2) \otul B(\Hil_1;\Hil_2) \to
 B(\noise_2;\noise_1) \otol B(\Hil_1;\Hil_2),
 \]
 given by linearisation of the bilinear map
 $
 (B, X) \mapsto B^\transpose \ot X
 $,
 is unbounded unless $B(\noise_1;\noise_2)$ or $B(\hil_1;\hil_2)$
 is finite-dimensional (see, for example~\cite{EfR}).
 We need to overcome this obstruction whilst tranposing a class of
 \emph{unbounded} operators.

 We exploit the fact that the transpose restricts to a unitary
 operator between the Hilbert-Schmidt classes, say
 $U: HS(\noise_1;\noise_2) \to
 HS(\noise_1;\noise_2)$
 and so,
 for any Hilbert spaces $\hil_1$ and $\hil_2$,
 there is a partial transpose
 \[
 U \ot I
 : HS(\noise_1;\noise_2) \ot HS(\hil_1;\hil_2) =
 HS(\noise_1\ot\hil_1;\noise_2\ot\hil_2)
  \to
  HS(\noise_2\ot\hil_1;\noise_1\ot\hil_2)
 \]
 which we denote by $H \mapsto H_\Transpose$.
 This is characterised by
 \begin{equation}
 \label{H transpose}
 \ip{c_1 \ot v_2}{H_\Transpose(c_2 \ot v_1)} =
 \ip{\ol{c_2} \ot v_2}{H (\ol{c_1} \ot v_1)}
 \qquad
 (c_i\in\noise_i, v_i\in\hil_i, i=1,2).
 \end{equation}
  The class of unbounded operators that we need to transpose
 is defined next.
 Recall the linear isomorphisms described in
 Proposition~\ref{when D =}.

 \begin{defn}
 The $(\noise_1,\noise_2)$-\emph{matrix space associated to}
 $(\Mil,\xi)$ is the following class of operators:
 \[
 \umatrix_{\noise_1;\noise_2}(\Mil, \xi):=
 \big\{
 T \in
 \Op_\Mil(\noise_1 \otul \Xi; \noise_2\ot\Hil):
 T E_\xi \in HS(\noise_1;\noise_2\ot\Hil)
 \big\},
 \]
 and for $T \in \umatrix_{\noise_1;\noise_2}(\Mil, \xi)$,
 its (\emph{matrix}) \emph{transpose} is given by
 \[
 T^\Transpose := \big( (T E_\xi)_\Transpose \big)^\xi,
 \]
 thus, for $B \in HS(\noise_1;\noise_2\ot\Hil)$,
  $(B^\xi)^\Transpose = (B_\Transpose)^\xi$.
 The corresponding \emph{column and row spaces} are given by
 \[
 \ucol_{\noise}(\Mil, \xi) := \umatrix_{\Comp;\noise}(\Mil, \xi)
 \text{ and }
 \urow_{\noise}(\Mil, \xi) :=
 \umatrix_{\noise;\Comp}(\Mil, \xi).
 \]
  \end{defn}

\begin{rems}
 This construction evidently enjoys the following properties:
 \begin{align*}
 &\umatrix_{\noise_1;\noise_2}(\Mil,\xi) =
 \big\{
  B^\xi: B
  \in HS(\noise_1;\noise_2\ot\Hil)
 \big\};
 \\
 &
 HS(\noise_1;\noise_2) \otul \Op_\Mil(\Xi) \subset
 \umatrix_{\noise_1;\noise_2}(\Mil,\xi)),
 \text{ with equality if }
 \Mil = \Comp, \text{ or if } \noise_1=\noise_2 = \Comp;
 \\
 &(H \ot X)^\Transpose = H^\transpose \ot X
 \qquad
 (H \in HS(\noise_1;\noise_2), X\in \Op_\Mil(\Xi));
 \\
 &
 \ucol_{\noise}(\Mil,\xi) = \Op_\Mil(\Xi; \noise\ot\Hil),
 \text{ whereas}
 \\
 &
 \urow_{\noise}(\Mil,\xi) =
 \big\{
 R \in
 \Op_\Mil(\noise\otul\Xi; \Hil):
 R E_\xi \in HS(\noise;\Hil)
 \big\};
  \\
 &\big(B(\noise_2;\noise_3) \otol \Mil\big)
 \umatrix_{\noise_1;\noise_2}(\Mil,\xi)
 \big( B(\noise_0;\noise_1)\ot I_\Hil \big)
 \subset
 \umatrix_{\noise_0;\noise_3}(\Mil,\xi).
  \end{align*}
 Moreover,
 $\umatrix_{\noise_1;\noise_2}(\Mil,\xi)$
 is a left $B(\noise_2) \otol \Mil$-module
 and a right $B(\noise_1)$-module, and the
 matrix transpose is characterised by
\[
\ip{c_1 \ot x'\xi}{T^\Transpose(c_2 \ot \xi)} =
 \ip{\ol{c_2} \ot x'\xi}{T(\ol{c_1} \ot \xi)}
 \qquad
 (c_1\in\noise_1, c_2 \in \noise_2, x'\in\Mil').
\]
\end{rems}

 We now need to relate the transpose operation
\[
 \umatrix_{\noise_1; \noise_2}(\Mil, \xi) \to
 \umatrix_{\noise_2; \noise_1}(\Mil, \xi),
 \quad
 T \mapsto T^\Transpose
\]
with the adjoint operation
\[
 \Opdagger_\Mil (\noise_1 \otul \Xi, \noise_2 \otul \Xi) \to
 \Opdagger_\Mil (\noise_2 \otul \Xi, \noise_1 \otul \Xi)
  \quad
 T \mapsto T^\dagger.
\]
 Specifically, we seek the appropriate space of operators/matrices
 compatible with both operations.
 To this end we define
\begin{multline*}
 HS_{\Mil,\xi}^\ddagger(\noise_1;\noise_2 \ot \Hil) :=
 \\
 \big\{ B \in
 HS(\noise_1; \noise_2 \ot \Hil)
 \cap
 B_{\Mil,\xi}^\ddagger(\noise_1;\noise_2 \ot \Hil):
 \
 B_\dagger \in HS(\noise_2;\noise_1 \ot \Hil)
 \big\}.
\end{multline*}
The proposition below justifies our choice.
 Its corollary, Theorem~\ref{Theorem 3.2} below, is key
 for the construction of quasifree stochastic integrals in
 Section~\ref{section: qf integrals}.
 For $i=1,2$, let $k_i$ denote the conjugations on $\noise_i$.

\begin{propn}
 \label{propn: HSdagger characterised}
 Let $B\in HS(\noise_1; \noise_2 \ot \Hil)$.
 Then the following are equivalent\tu{:}
 \begin{rlist}
 \item
 $B \in HS_{\Mil,\xi}^\ddagger(\noise_1;\noise_2 \ot \Hil)$.
 \item
 $B_\Transpose \in HS_{\Mil,\xi}^\ddagger(\noise_2;\noise_1 \ot \Hil)$.
 \item
 $\Ran B \subset \Dom k_2 \ot S_\xi$ and
 $\ol{B} := (k_2 \ot S_\xi) B k_1 \in HS(\noise_1; \noise_2 \ot \Hil)$.
 \end{rlist}
 In this case,
 \begin{equation*}
 B_{\dagger \Transpose} =
 B_{\Transpose \dagger} = \ol{B}.
 \end{equation*}
 \end{propn}
\begin{proof}
 For all $c_1\in\noise_1$, $c_2\in\noise_2$ and $x'\in\Mil'$,
 \[
 \ip{B_\Transpose c_2}{c_1 \ot x'^* \xi}
 =
  \ip{B k_1 c_1}{(k_2\ot F_\xi)(c_2 \ot x' \xi)}.
 \]
 Since
 $\noise_2\otul \Xi$ is a core for $k_2\ot F_\xi$ and
 $(k_2\ot F_\xi)^* = k_2\ot S_\xi$,
 it follows
 (using the characterisation~\eqref{characterised})
 that (ii) and (iii) are equivalent,
 and also that when they hold,
 $\ol{B} = B_{\Transpose \dagger}$.

 If (i) holds
 then,
 for all $c_1\in\noise_1$, $c_2\in\noise_2$ and $x\in\Mil'$,
 \begin{align*}
 \ip{B_\Transpose c_2}{c_1 \ot x^* \xi}
 &=
 \ip{B \ol{c_1}}{\ol{c_2} \ot x^* \xi}
 \\
 &=
 \ip{\ol{c_1}\ot x \xi}{B_\dagger \ol{c_2}}
 =
 \ip{c_2 \ot x \xi}{B_{\dagger \Transpose} c_1}
 \end{align*}
 so, by the characterisation~\eqref{characterised},
 \[
 B_\Transpose \in
 B_{\Mil,\xi}^\ddagger(\noise_2;\noise_1 \ot \Hil)
 \
 \text{ and } \
  B_{\Transpose \dagger} =
 B_{\dagger \Transpose} \in
 HS(\noise_2; \noise_1 \ot \Hil).
 \]
 Thus (i) implies (ii), and since $B_{\Transpose \Transpose} = B$,
 also (ii) implies (i).
 This completes the proof.
\end{proof}

 \begin{defn}
 The
 $(\noise_1,\noise_2)$-\emph{adjointable matrix space associated to}
 $(\Mil,\xi)$ is the class of operators defined by
 \begin{multline*}
 \umatrixdagger_{\noise_1;\noise_2}(\Mil, \xi) :=
 \\
 \big\{
 T \in
 \umatrix_{\noise_1, \noise_2}(\Mil,\xi)
 \cap
 \Opdagger_\Mil(\noise_1\otul \Xi, \noise_2\otul \Xi):
 T^\dagger \in
 \umatrix_{\noise_2, \noise_1}(\Mil,\xi)
 \big\}
 \end{multline*}
 The corresponding \emph{column and row spaces} are given by
\[
 \ucoldagger_{\noise}(\Mil, \xi) := \umatrixdagger_{\Comp;\noise}(\Mil, \xi)
 \ \text{ and } \
  \urowdagger_{\noise}(\Mil, \xi) :=
 \umatrixdagger_{\noise;\Comp}(\Mil, \xi).
 \]
 \end{defn}

 We now have a matrix space of affiliated operators
 having adjoints and transposes;
 the key properties are summarised next.

 \begin{thm}
 \label{Theorem 3.2}
 The following hold
 \begin{align*}
 &\umatrixdagger_{\noise_1;\noise_2}(\Mil, \xi)=
 \big\{
 B^\xi: B \in HS^\ddagger_{\Mil,\xi}(\noise_1;\noise_2\ot\Hil)
 \big\}
 \\
 &
 \qquad\qquad\quad\
 =
 \big\{
 T\in \Opdagger_\Mil(\noise_1 \ot \Xi, \noise_2 \otul \Xi):
 T E_\xi \in HS(\noise_1; \noise_2\ot\Hil),
 T^\dagger E_\xi \in HS(\noise_2; \noise_1\ot\Hil)
 \big\}
 ;
 \\
 &
 HS(\noise_1;\noise_2) \otul \Opdagger_\Mil(\Xi) \subset
 \umatrixdagger_{\noise_1;\noise_2}(\Mil,\xi)),
 \text{ with equality if }
 \Mil = \Comp, \text{ or if } \noise_1=\noise_2 = \Comp;
 \\
 &(H \ot X)^{\Transpose \dagger} =
 H^{\transpose \dagger} \ot X^\dagger =
 \ol{H} \ot X^\dagger
 \qquad
 (H \in HS(\noise_1;\noise_2), X\in \Opdagger_\Mil(\Xi));
 \\
 &
 \ucoldagger_{\noise}(\Mil,\xi) =
 \big\{
 C \in
 \Opdagger_\Mil(\Xi, \noise\otul\Xi):
 C^\dagger E_\xi \in HS(\noise;\Hil)
 \big\}
 \text{ \tu{(}restoring symmetry with\tu{)}}
 \\
 &
 \urowdagger_{\noise}(\Mil,\xi) =
 \big\{
 R \in
 \Opdagger_\Mil(\noise\otul\Xi, \Xi):
 R E_\xi \in HS(\noise;\Hil)
 \big\}
 \text{ \tu{(}but also\tu{)}}
 \\
 &\ucoldagger_{\noise}(\Mil,\xi) =
 \big\{
 \ket{\zeta}^\xi:
  \zeta \in \Dom k \ot S_\xi
 \big\}
 =
 \big\{
 C \in \Op_\Mil(\Xi; \noise\ot\Hil): C \xi \in \Dom k \ot S_\xi
 \big\}.
\end{align*}
Moreover, for all
 $T \in \umatrixdagger_{\noise_1;\noise_2}(\Mil, \xi)$
  and $C \in \ucoldagger_\noise(\Mil, \xi)$,
\begin{align}
 &T^\dagger, T^\Transpose \in
 \umatrixdagger_{\noise_2;\noise_1}(\Mil, \xi),
 \quad
 T^{\Transpose \Transpose} = T^{\dagger \dagger} = T,
 \quad
 T^{\dagger \Transpose} = T^{\Transpose \dagger},
 \nonumber
 \\
 &T^{\dagger \Transpose} E_\xi =
 (k_2\ot S_\xi) T E_\xi k_1,
 \ \text{ and } \
  C^{\dagger \Transpose} \xi = (k \ot S_\xi) C \xi.
 \label{M c xi}
\end{align}
\end{thm}
\begin{rems}
 Note further that
\[
 \big(B(\noise_2;\noise_3) \ot I_\Hil\big)
 \umatrixdagger_{\noise_1;\noise_2}(\Mil,\xi)
 \big( B(\noise_0;\noise_1)\ot I_\Hil \big)
 \subset
 \umatrixdagger_{\noise_0;\noise_3}(\Mil,\xi),
 \]
 $\umatrixdagger_{\noise_1;\noise_2}(\Mil,\xi)$
 is a left $B(\noise_2)$-module
 and a right $B(\noise_1)$-module.
\end{rems}

 The relationship between the various spaces is
 seen in the following commutative diagram, in which
 the horizontal arrows represent linear isomorphisms
 and all other arrows represent inclusions.
\[
 \xymatrix{
 \Op_\Mil(\noise_1 \otul \Xi; \noise_2 \ot \Hil)
 \ar[r] &
  \Op(\noise_1; \noise_2 \ot \Hil)
 &
 \\
   \Opdagger_\Mil(\noise_1 \otul \Xi, \noise_2 \otul \Xi)
  \ar[u]
  \ar[r] &
  B^\ddagger_{\Mil,\xi}(\noise_1; \noise_2\ot\Hil)
  \ar[u]
 &
 \\
  \umatrixdagger_{\noise_1,\noise_2}(\Mil,\xi)
 \ar[u]
 \ar[d]
 \ar[r] &
 HS^\ddagger_{\Mil,\xi}(\noise_1; \noise_2\ot\Hil)
 \ar[u]
 \ar[d]
 &
 \\
  \umatrix_{\noise_1, \noise_2}(\Mil,\xi)
  \ar@/_-4pc/[uuu]
    \ar[r] &
   HS(\noise_1; \noise_2\ot\Hil)
    \ar@/_4pc/[uuu]
   }
\]

 We end this section by introducing a transform
 between matrices and columns which is
 one of the ingredients of the construction of
 quasifree integrals in Section~\ref{section: qf integrals}.
 Denote by
 $\pi$ the sum-flips on both $\noisetwo$ and
 $\noisetwo \ot \Hil = (\noise \ot \Hil)^{\op 2}$,
 set $\khat := \Comp \op \noise$,
 \begin{align*}
 &\umatrix_{\khat}(\Mil, \xi)_0 :=
 \Big\{
 \left[\begin{smallmatrix} & R \\ C & \end{smallmatrix}\right]:
 C \in \ucol_\noise(\Mil,\xi) \text{ and }
 R \in \urow_\noise(\Mil,\xi)
 \Big\}, \text{ and }
 \\
 &\umatrixdagger_{\khat}(\Mil, \xi)_0 :=
 \umatrix_{\khat}(\Mil, \xi)_0 \cap\umatrixdagger_{\khat}(\Mil,\xi),
 \end{align*}
  and set
  $k^\pi := (k \oplus k) \circ \pi$.
  \begin{cor}
 \label{Corollary 3.3}
 The map
 \[
 \umatrix_{\khat}(\Mil, \xi)_0 \to \ucol_{\noisetwo}(\Mil,\xi) =
 \Op_\Mil(\Xi; \noisetwo \ot \Hil),
 \quad
 T =
  \begin{bmatrix} & R \\ C & \end{bmatrix}
 \mapsto
 \colT :=
 \begin{bmatrix} C \\ R^\Transpose \end{bmatrix}
 \]
 is a linear isomorphism which restricts to an isomorphism
 $\umatrixdagger_{\khat}(\Mil,\xi)_0 \to
 \ucoldagger_{\noisetwo}(\Mil,\xi)$
 satisfying
 $\colTdagger = \pi \circ \colTconjugate$, and thus
  \[
 \colTdagger \xi =
 ( k^\pi \ot S_\xi ) \colT \xi.
 \]
 \end{cor}

\section{It\^o integral and commutation relations}
 \label{section: Ito commutation relations}

 In this section we prove a commutation relation between
 second quantisation and the abstract It\^o integral.
 First we set up notation for stochastic analysis in
 Fock space.
 \emph{Fix a Hilbert space $\init$ and
 a separable Hilbert space $\noise$}.
 For a subinterval $I$ of $\Rplus$, set
 \[
 \Kil_I = L^2(I;\noise),\
 \Fock_{\noise, I} = \Gamma(\Kil_I), \
 \initFock_{\noise, I} = \init \ot \Fock_{\noise, I},
 \text{ and }
 \Omega_{\noise, I} = (1, 0, 0, \cdots ) \in \Fock_{\noise, I},
 \]
 dropping the $I$ when it is all of $\Rplus$.
 The tensor decompositions
 \[
 \initFocknoise = \initFock_{\noise, {[0,s[}} \ot
 \Fock_{\noise, [s,t[} \ot
 \Fock_{\noise, [t,\infty[}
 \qquad
 (0\leq s \leq t \leq \infty)
 \]
 are witnessed by exponential vectors.
 Write
  \begin{equation}
 \label{pt and Pt}
 p_t \text{ for } M_{1_{[0,t[}} \text{ on } \Kil\
 \text{ and } \
 P_t \text{ for }
 I_\hil \ot \Gamma(p_t) \text{ on } \hil \ot \Fock_\noise
 \qquad
 (t\geq 0),
 \end{equation}
 where
 $M$ denotes multiplication operator and
 $\hil$ can be $\Comp$, $\init$ (or $\noise \ot \init$),
 depending on context,
 and let
 $\Kil_t$, $\Fock_{\noise, t}$ and $\initFock_{\noise, t}$ be the
 images of the respective orthogonal projections.
 Then
 $\Kil \ot \initFocknoise = L^2(\Rplus; \noise \ot \initFocknoise)$
 and, by Fubini's Theorem,
 \begin{align}
 &\big\{
 y \in \Kil \ot \initFocknoise:
 \text{ for a.a. } t\in\Rplus,
 y_t = y_{t)} \ot \Omega_{\noise, [t,\infty[}
 \text{ for some } y_{t)} \in \initFock_{\noise, {[0,t[}}
 \big\}
 \ \text{and}
 \nonumber
  \\
  &\big\{
 y \in \Kil \ot \initFocknoise:
 \forall_{t\ge 0} \,
 (p_t \ot I_{\initFocknoise}) y \in \Kil \ot \initFock_{\noise, t}
 \big\},
   \label{L2Omega characterisation}
  \end{align}
 coincide; the common subspace is called
 the \emph{$\Omega$-adapted subspace} of $\Kil \ot \initFocknoise$,
 and is denoted
 $L^2_\Omega\big(\Rplus; \noise \ot \initFocknoise\big)$.
 Let $V_\Omega$ denote the inclusion
 $L^2_\Omega\big(\Rplus; \noise \ot \initFocknoise\big)
 \to \Kil \ot \initFocknoise$,
 and $P_\Omega$ the orthogonal projection $V_\Omega V_\Omega^*$.
 Recall the gradient operator defined in the introduction
 and the convention on ampliation. The operator
 $V_\Omega^* \nabla$ is bounded and $D := \ol{V_\Omega^* \nabla}$
 is a surjective partial isometry with kernel $\initFock_{\noise, 0}$,
 which is called the \emph{adapted gradient operator} (\cite{AtL}).
 The It\^o integral is the isometry
 \[
 \Ito := D^* = \Skorohod V_\Omega:
  L^2_\Omega\big(\Rplus; \noise \ot \initFocknoise\big)
  \to \initFocknoise;
  \]
 the divergence operator $\Skorohod := \nabla^*$ being an abstract
 Hitsuda-Skorohod integral
 (\cite{LQSI}).
 We further define
 \[
 L^2_{\Omega, \loc} \big(\Rplus; \noise \ot \initFocknoise\big) :=
 \big\{
 y \in
 L^2_{\loc} \big(\Rplus; \noise \ot \initFocknoise\big):
 \forall_{t\ge 0} \
 y_{[0,t[} \in \Kil \ot \initFock_{\noise, t}
 \big\},
 \]
 and for $t\in\Rplus$ and
 $z\in L^2_{\Omega, \loc} \big(\Rplus; \noise \ot \initFocknoise\big)$,
 $\Ito_t y := \Ito y_{[0,t[}$.
 The following
 characterisation of operators affiliated to the
 von Neumann algebra
 $L^\infty(\Rplus) \otol B(\noise)$
 is useful.

 \begin{lemma}
 \label{invariance = affiliation}
 Let $T$ be a closed and densely defined operator on $\Kil$.
 Then the following are equivalent.
 \begin{rlist}
 \item
 $T$ is affiliated to $L^\infty(\Rplus) \otol B(\noise)$.
 \item
 $T$ satisfies the invariance condition
 \begin{equation}
 \label{Dom T pt invariant}
 T p_t \supset p_t T
 \qquad
 (t\geq 0).
 \end{equation}
 \item
 $T$ is `pointwise adjointable', that is
 for all $f\in \Dom T^*$ and $g\in\Dom T$,
 \[
 \ip{f(t)}{(Tg)(t)} = \ip{(T^*f)(t)}{g(t)}
 \qquad
 \text{ for a.a. } t \geq 0.
 \]
 \end{rlist}
 \end{lemma}
\begin{proof}
 Since, for all $t\geq 0$,
 $p_t \in L^\infty(\Rplus) \ot I_\noise$,
 the commutant of
 $L^\infty(\Rplus) \otol B(\noise)$,
 (i) implies (ii).
 On the other hand,
 viewing $L^\infty(\Rplus)$ as the dual of $L^1(\Rplus)$,
 for $f\in \Dom T^*$ and $g\in\Dom T$ the set
 \[
 \big\{
 \varphi \in L^\infty(\Rplus): \norm{\varphi}_\infty \leq 1
 \text{ and }
 \ip{T^* f}{\varphi \cdot g} = \ip{f}{\varphi \cdot T g}
 \big\}
 \]
 is compact and metrizable in the relative weak topology,
 and step functions with $L^\infty$-bound at most one
 are dense in the unit ball of $L^\infty(\Rplus)$.
 It follows that (ii) implies (i).

 The equivalence of (i) and (iii) is
 evident from the identities
 \begin{align*}
 &\int dt \varphi(t) \ip{f(t)}{(Tg)(t)}
 =
 \ip{f}{\varphi \cdot T g},
 \text{ and }
 \\
 &\ip{T^*f}{\varphi \cdot  g}
 =
 \int dt \varphi(t) \ip{(T^* f)(t)}{(g(t)},
 \end{align*}
 for $f\in \Dom T^*$ and $g\in\Dom T$
 and $\varphi \in L^\infty(\Rplus)$.
\end{proof}

\begin{rem}
 A good reference for the identification of
 $L^\infty(\Rplus) \otol \Mil$ and $L^\infty(\Rplus; \Mil)$,
 for a von Neumann algebra $\Mil$ with separable predual,
 is Theorem 1.22.13 of~\cite{Sakai}.
\end{rem}

 \begin{lemma}
 \label{lemma: tensor Omega}
 Let $R = T \ot X$, where $T$ and $X$ are closed densely defined operators
 on $\Kil$ and $\initFocknoise$ respectively,
 satisfying
 \begin{align*}
 &T\, \affiliated L^\infty(\Rplus) \otol B(\noise),
 \ \text{ equivalently }
 T p_t \supset p_t T,
 \ \text{ and }
 \\
 &
 X\big( \initFock_{\noise, t} \cap \Dom X \big) \subset \initFock_{\noise, t},
 \ \text{ equivalently } \
 X P_t = P_t X P_t
 \qquad
 (t \geq 0).
 \end{align*}
  Then
 \begin{align*}
 &(T\ot X) \big(
  L^2_\Omega(\Rplus; \noise \ot \initFocknoise)\ \cap\ \Dom\, T \ot X
 \big)
 \subset
 L^2_\Omega(\Rplus; \noise \ot \initFocknoise),
  \ \text{equivalently }
  \\
  &(T \ot X) P^\Omega = P^\Omega (T \ot X) P^\Omega.
 \end{align*}
\end{lemma}
\begin{proof}
 Set $I := I_{\initFocknoise}$.
 By Part (f) of Proposition~\ref{Lemma 3} and Corollary~\ref{tensor invariance},
  we have
 \begin{alist}
 \item
  $R(p_t \ot I) \supset (p_t \ot I) R$, and
 \item
  $R \big( \Kil \ot \initFock_{\noise, t} \cap \Dom R \big)
  \subset \Kil \ot \initFock_{\noise, t}$,
  for all $t\geq 0$.
 \end{alist}
 Let
 $z \in L^2_\Omega(\Rplus; \noise \ot \initFocknoise) \cap \Dom R$ and
 $t\geq 0$.
 By adaptedness and (a),
 $(p_t \ot I) z \in (\Kil \ot \initFock_{\noise, t}) \cap \Dom R$ and
 \[
 R(p_t \ot I) z = (p_t \ot I) R z
 \]
 so, by (b), $R( p_t \ot I) z \in \Kil \ot \initFock_{\noise, t}$
 and thus $(p_t \ot I) R z \in \Kil \ot \initFock_{\noise, t}$.
 Therefore, by~\eqref{L2Omega characterisation},
  $R z\in L^2_\Omega(\Rplus; \noise \ot \initFocknoise)$,
 as required.
\end{proof}

\begin{notn}
 For operators $T$ and $X$ of the above form we set
 \begin{equation}
 \label{tensor Omega}
 T \ot_\Omega X := V_\Omega^* (T \ot X) V_\Omega
 \end{equation}
 where $V_\Omega$ is the inclusion map
 $L^2_\Omega(\Rplus; \noise \ot \initFocknoise)
 \to
 \Kil \ot \initFocknoise$.
 \end{notn}
 \begin{rem}
 Operators of the form $T \ot_\Omega X$ are closed,
 as is easily verified.
 \end{rem}

 The next two results involve the (ampliated) gradient operator
 on Fock space (which is defined in the introduction),
 and the second quantised operators of
 Proposition~\ref{second quantisation}.

\begin{lemma}
\label{gradient second quantisation commutation}
 Let $A$ and $T$ be closed densely defined operators on
 $\init$ and $\Kil$ respectively.
 Then
 \[
 \nabla \big( A \otul \Gamma(T)_| \big) \subset
 \big( T \otul A \otul \Gamma(T)_| \big) \nabla.
 \]
\end{lemma}
 \begin{proof}
 For $v \in \Dom A$ and $g \in \Dom T$,
 \begin{align*}
 &v \ve(g) \in \Dom \nabla,
 Av \ot \ve(Tg) \in \Dom \nabla,
 \\
 &\nabla v \ve(g) =
 g \ot v \ot \ve(g) \in
 \Dom \big( T \otul A \otul \Gamma(T)_| \big),
 \ \text{ and}
 \\
 &\big( T \ot A \ot \Gamma(T) \big) \nabla v \ve(g) =
 Tg \ot Av \ot \ve(Tg) =
 \nabla \big( Av \ot \ve(Tg) \big).
 \end{align*}
 The result follows.
 \end{proof}

With these we are able to establish
 a key commutation relation between
 the operations of second quantisation and It\^o integration.

\begin{thm}
 \label{Ito second quantisation commutation}
 Let $X = A \ot \Gamma(T)$ where
 $A$ and $T$ are closed densely defined operators on
 $\init$ and $\Kil$ respectively,
 with $T$ affiliated to $L^\infty(\Rplus) \otol B(\noise)$.
 Then
 \[
  X\, \Ito = \Ito \circ (T \ot_\Omega X)
 \]
 and, for any core
 $\mathcal{C}$ for $X$,
 $D (\mathcal{C})$ is a core for $X\Ito$.
\end{thm}

\begin{proof}
 The strategy of proof is as follows.
 We prove successively:
\begin{alist}
\item
 For all $t\geq 0$,
 $X \big( \initFock_{\noise, t} \cap \Dom X \big)
 \subset \initFock_{\noise, t}$.
 \item
   $X\, \Ito \supset \Ito \circ (T \ot_\Omega X)$.
 \item
 The operators $X \Ito$ and $\Ito \circ (T \ot_\Omega X)$
 are both closed.
\item
If $\mathcal{C}$ is a core for $X$ then $D (\mathcal{C})$ is a core
for $X\Ito$.
\item
 Setting $\mathcal{D} := \Dom T \otul \Dom A \otul \Exps(\Dom T)$,
 we have
 \[
 P^\Omega (\mathcal{D}) \subset \Dom T \ot X.
 \]
\end{alist}
 Then, setting $\mathcal{C} = \Dom A \otul \Exps(\Dom T)$,
 we have
 \[
 V_\Omega D(\mathcal{C}) \subset
 P^\Omega \mathcal{D} \subset
 \Dom T \ot X.
 \]
  Thus, by (d),
  $D (\mathcal{C})$ is a core for $X\Ito$
  contained in $\Dom (T \ot_\Omega X)$,
  which equals $\Dom \big( \Ito \circ (T \ot_\Omega X) \big)$.
 Since $\Ito \circ (T \ot_\Omega X)$ is closed,
  it follows that
  the inclusion in (b) is an equality
    and the proof will then be complete.

 (a)
 Let $t \geq 0$.
 First note that $T p_t =p_t T p_t$.
 To see this
 use Lemma~\ref{invariance = affiliation} and
 observe that, for  $f \in \Dom T p_t$,
 \[
 p_t f \in \Dom T = \Dom p_t T \subset \Dom T p_t
 \ \text{ and } \
 T p_t f = T p_t p_t f = p_t T p_t f.
\]
 Now let
 $\zeta \in \Fock_{\noise, t} \cap \Dom \Gamma(T)$.
 By Proposition~\ref{second quantisation}, we have
 \begin{align*}
 \Gamma(T) \zeta = \Gamma(T) \Gamma(p_t) \zeta
 &= \Gamma(T p_t) \zeta
 \in \Ran \Gamma(p_t T p_t) \subset \Ran \Gamma(p_t) = \Fock_t.
 \end{align*}
 Thus
 $\Gamma(T) \big( \Fock_{\noise, t} \cap \Dom \Gamma(T) \big)
 \subset \Fock_{\noise, t}$,
 and Corollary~\ref{tensor invariance} implies that
 $X \big( \initFock_{\noise, t} \cap \Dom X \big)
 \subset \initFock_{\noise, t}$, as required.

 (b)
 By (a),
 Lemma~\ref{lemma: tensor Omega} applies,
 thus
 \begin{equation}
 \label{T X V}
 (T \ot X) V_\Omega = P^\Omega (T \ot X) V_\Omega
 \end{equation}
 and we may form
 the operator $T \ot_\Omega X$.
 Let
 \[
 z \in \Dom T \ot_\Omega X
 \ \text{ and } \
 \zeta \in \Dom A^* \otul \Exps(\Dom T^*).
 \]
 Then,
 by Lemma~\ref{gradient second quantisation commutation}
 and Proposition~\ref{second quantisation},
 \begin{align*}
 \ip{\zeta}{\Ito\big( (T \ot_\Omega X) z \big)}
 &=
 \ip{\nabla \zeta}{ V_\Omega V_\Omega^* (T \ot X) V_\Omega z}
 \\
 &=
 \ip{\nabla \zeta}{\big( T \ot A \ot \Gamma(T) \big) V_\Omega z}
 \\
 &=
  \ip{\nabla \big( A^* \ot \Gamma(T)^* \big) \zeta}{ V_\Omega z}
 =
 \ip{\big( A^* \ot \Gamma(T)^* \big) \zeta}{ \Ito z}.
 \end{align*}
 Since $\Dom A^* \otul \Exps(\Dom T^*)$
 is a core for $A^* \ot \Gamma(T)^* = X^*$,
 this implies that $\Ito z \in \Dom X$ and
 $X\, \Ito z =  \Ito \big( (T \ot_\Omega X) z \big)$.
 This proves (b).

 (c)
 Being a closed operator composed with a bounded operator,
 $X \Ito$ is closed (Lemma~\ref{Lemma 1}).
 To see that $R:= \Ito \circ (T \ot_\Omega X)$ is closed too,
 let $(z_n)$ be a sequence in $\Dom R = \Dom (T\ot X)V_\Omega$
 such that
 $
 z_n \to z$ and $R z_n \to w$.
 Then
  $V_\Omega z_n \to V_\Omega z$
  and, by~\eqref{T X V},
 \begin{align*}
 (T \ot X) V_\Omega z_n
 &=
 P^\Omega (T \ot X) V_\Omega z_n
 \\
 &=
 V_\Omega D\Ito(T \ot_\Omega X) z_n
 =
 V_\Omega D R z_n \to V_\Omega D w.
 \end{align*}
 Therefore, since $T\ot X$ is closed,
 $V_\Omega z \in \Dom T \ot X$ and
 $(T\ot X) V_\Omega z = V_\Omega D w$.
 Thus,
 since $w \in \Ran \Ito$,
 $z \in \Dom (T \ot X) V_\Omega = \Dom R$ and
 \[
 R z =
  \Ito V_\Omega^* V_\Omega D w =
 \Ito D w = w.
\]
 Thus $R$ is closed too.

 (d)
 This follows from Part (c) of Lemma~\ref{Lemma 1}
 since $X$ is closed, $\Ito$ is isometric
 $\Ito D = \Ito \Ito^* = I_\init \ot \Gamma(0)$,
 and
 the evident inclusion
 \[
 \big( I_\init \otul \Gamma(0) \big) \, A^* \otul \Gamma(T^*) \subset
 A^* \otul \Gamma(T^*) \, \big( I_\init \otul \Gamma(0) \big)
 \]
 implies that $X \Ito \Ito^* \supset \Ito \Ito^* X$,
 by the adjoint-product-inclusion relation and
 Proposition~\ref{second quantisation}.

 (e)
 Let
 $\zeta = f^1 \ot u \ot \ve(f^2)$ and
 $\eta = g^1 \ot v \ot \ve(g^2)$,
 where
 $f^1, f^2 \in \Dom T^*$, $u\in \Dom A^*$,
 $g^1, g^2 \in \Dom T$ and $v\in \Dom A$.
 Then, by Lemma~\ref{invariance = affiliation},
 \begin{align*}
 &\ip{(T\ot X)^* \zeta}{P^\Omega \eta}
 \\
 &\qquad \qquad
 =
 \int dt \ip{(T^*f^1)(t)\ot A^*u \ot \ve(T^*f^2)}{g^1(t)\ot v \ot \ve(p_t g^2)}
 \\
 &\qquad \qquad
 =
 \int dt \ip{f^1(t)\ot u \ot \ve(f^2)}{(Tg^1)(t)\ot Av \ot \ve(p_t T g^2)}
 =
\ip{\zeta}{P^\Omega (T \ot X) \eta}.
 \end{align*}
 Thus $(T \ot X) P^\Omega \supset P^\Omega (T \ot X)$,
 in particular $P^\Omega (\mathcal{D}) \subset \Dom T \ot X$.
\end{proof}

 \begin{rems}
 For comparison,
 note that if $X$ is bounded
 (equivalently, if $A$ is bounded and $T$ is a contraction) then
 \[
 \ol{X \Skorohod} = \Skorohod \circ (T \ot X),
 \]
 but $X \Skorohod$ is typically not closed (e.g.\ $T=0$).

 We shall use this result with $A$ and $T$ being conjugate-linear
 operators.
 \end{rems}

 \begin{cor}
 \label{DP =MD}
  For all $t\geq 0$,
  \begin{equation}
  \label{M Omega t}
 D P_t = M_t^\Omega D
 \ \text{ where } \
 M_t^\Omega :=
 p_t \ot_\Omega I
 \ \text{ and } \
 I = I_{\noise \ot \initFocknoise}.
  \end{equation}
   \end{cor}
\begin{proof}
 Let $t\geq 0$.
 In view of the identity
 $(p_t \ot P_t) V_\Omega = (p_t \ot I) V_\Omega$,
 the theorem implies that
 $P_t \Ito =
 \Ito\big( p_t \ot_\Omega P_t \big) =
 \Ito\big( p_t \ot_\Omega I \big)
 $,
 and~\eqref{M Omega t} follows on taking adjoints.
\end{proof}

\section{CCR algebras and quasifree states}
 \label{section: CCR algebras and qf states}

 For any nondegenerate symplectic space $(V,\sigma)$ there is
 an associated simple $C^*$-algebra, denoted $CCR(V,\sigma)$; it
 is generated by elements $\{ w_v: v \in V \}$ satisfying
 the canonical commutation relations in Weyl form:
 \[
 w_u w_v = e^{-i\sigma(u,v)} w_{u+v}
 \text{ and }
 w_u^* = w_{-u}
 \qquad
 (u,v\in V).
 \]
 Every *-algebra morphism
 from $CCR_0(V,\sigma) := \Lin\{ w_v: v \in V \}$ to a
 $C^*$-algebra $\Al$, extends uniquely to a $C^*$-morphism from
 $CCR(V,\sigma)$ to $\Al$,
 and every symplectic map $R$ from $V$ into another nondegenerate
 symplectic space $V'$ induces a $C^*$-monomorphism
 $\phi_R: CCR(V,\sigma) \to CCR(V',\sigma')$ satisfying
 $\phi_R (w_v) = w_{Rv}$ ($v\in V$)
 (\cite{slawny},~\cite{Man};
 see  Theorem 5.2.8 of~\cite{BrR},
 and Chapter 2 of~\cite{petz}).
 When $(V',\sigma') = (V,\sigma)$
 and $R$ is a symplectic automorphism,
 $\phi_R$ is known as a Bogoliubov transformation.
 Typically $V$ is a real subspace of a complex Hilbert space
 and $\sigma = \im \ip{\cdot}{\cdot}$
 (in this case we write $CCR(V)$);
 when $V$ is a complex subspace,
 the \emph{guage transformations} of $CCR(V)$ are the Bogoliubov transformations
 $\phi_z$ induced by the symplectic automorphisms
 $v \mapsto z v$ ($z \in \mathbb{T}$).
 The \emph{characteristic function} of a state $\varphi$ on
 $CCR(V,\sigma)$ is the complex-valued function
 $\wh{\varphi} := \varphi \circ w$ on $V$.
 Given any nonnegative quadratic form $\mathfrak{a}$
 on $V$ satisfying
 \begin{equation*}
 \sigma(u,v)^2 \leq
 \mathfrak{a}[u] \mathfrak{a}[v]
 \qquad
 (u,v\in V),
 \end{equation*}
 there is a unique state $\varphi$ on $CCR(V,\sigma)$
 whose characteristic function is given
 by
\begin{equation}
 \label{qf characteristic function}
 \wh{\varphi}: v \mapsto e^{-\frac{1}{2}\mathfrak{a}[v]}
 \end{equation}
 (see \cite{petz}, Theorem 3.4).
 Such states are called (mean zero) \emph{quasifree} states.
 When $V$ is a complex subspace of a Hilbert space,
 a state $\varphi$ on $CCR(V)$ is called guage-invariant if
 it is invariant under the group of guage transformations.
 Thus the above quasifree state is guage invariant if
 its covariance satisfies
 $\mathfrak{a}[ z v] = \mathfrak{a}[v]$ ($v\in V$, $z\in \mathbb{T}$).
 Quasifree states are obviously \emph{regular}, that is
 $t \in \Real \mapsto \wh{\varphi}(t v) \in \Comp$ is continuous
 for all $v\in V$. As a consequence their GNS
 representations yield field operators $R_\varphi(v)$
 as Stone-generators of the unitary group
 $\big(\pi_\varphi(w_{tv})\big)_{t\in \Real}$ and thus,
 when $(V,\sigma)$ is a
 \emph{complex}
 subspace of
 $(\Hil, \im \ip{\cdot}{\cdot})$
 for a complex Hilbert space $\Hil$, also
  annihilation and creation operators
 $a_\varphi(v):=
 \frac{1}{2}\big( R_\varphi(v) + i R_\varphi(i v) \big)$,
 respectively
 $a^*_\varphi(v):=
 \frac{1}{2}\big( R_\varphi(v) - i R_\varphi(i v) \big)$
 ($v\in V$).
 The latter are fully formed closed mutually adjoint operators
 satisfying the canonical commutation relations in the form
 \[
 \norm{a^*_\varphi(v)\zeta}^2  - \norm{a_\varphi(v)\zeta}^2 =
 \norm{v}^2 \norm{\zeta}^2
 \qquad
 \big( \zeta \in \Dom a^*_\varphi(v) = \Dom a_\varphi(v) \big)
 \]
 (\cite{BrR}, Lemma 5.1.12). \emph{Warning}: We use the
 probabilists'  normalisation rather than
 that of the mathematical physicists.
 The case where
 $(V, \sigma) =
 (\Hil, \im \ip{\cdot}{\cdot})$
 and
 $\mathfrak{a} = \norm{\cdot}^2$,
 for a complex Hilbert space
 $\Hil$,
 is the Fock state.
 Its GNS representation is given by the Fock-Weyl operators
 defined in the introduction and Fock vacuum vector.
 For any nondegenerate symplectic space $(V,\sigma)$ and
 symplectic map $R: V \to \Hil$
 satisfying
 $|\sigma(u,v)| \leq \norm{R u }\norm{R v}$
 ($u, v \in V$),
 there is a
 representation $\pi_R$ of
 $CCR(V,\sigma)$ on $\Gamma(\Hil)$
 satisfying
 $\pi_R(w_v) = W_0 (R v)$
 and a quasifree state with
 characteristic function~\eqref{qf characteristic function}
 in which $\mathfrak{a}[v] = \norm{R v }^2$ ($v\in V$).
 There is an extensive literature on quasifree states;
 the notes~\cite{petz} are useful,
 and~\cite{BrR} provides their context in quantum statistical mechanics.

 \begin{rem}
 The analogue of quasifree states in free probability is
 investigated in~\cite{shl}.
 \end{rem}

  A pair $(H_1, H_2)$, consisting of closed subspaces of a real Hilbert
 space, is said to be \emph{in generic position} if
 $H_1 \cap H_2$,
 $H_1^{\perp} \cap H_2$,
 $H_1 \cap H_2^{\perp}$ and
 $H_1^{\perp} \cap H_2^{\perp}$ are all trivial (\cite{halmos}).
 Araki's Duality Theorem, which we quote next,
 is central to the understanding of von Neumann algebras
 associated with quasifree states of CCR algebras.

 \begin{thm}[{[$\text{Ar}_{1,2}$]}]
 \label{duality theorem}
 Let $H_1$ and $H_2$ be closed real subspaces
 of a complex Hilbert space $\Hil$.
 Suppose  that
 $(H_1, H_2)$ is in generic position and let
 $\pi$ be the Fock representation of $CCR(\Hil)$.
 For $i=1,2$, let $\pi_i = \pi \circ \phi_i$ where $\phi_i$
 is the natural $C^*$-monomorphism $CCR(H_i) \to CCR(\Hil)$, then
 $\pi_i$ is a faithful, irreducible representation which generates a
 Type III factor $\Noise_i$ for which the Fock vacuum
 $\Omega_{\Hil}$ is cyclic and separating
 and $\Noise_2 = (\Noise_1)'$.
 \end{thm}

 \emph{In this section} $\Hil = \Kiltwo$
 where $\Kil$ is the complexification of a real Hilbert
 space.
 Viewing $\Kil$ and $\Kiltwo:= \Kil \op \Kil$ as real vector spaces,
 they carry  the symplectic forms
 $\im \ip{\cdot}{\cdot}_\Kil$ and
 $\im \ip{\cdot}{\cdot}_{\Kiltwo}$ respectively,
 and the real inner products
 $\re \ip{\cdot}{\cdot}_\Kil$ and
 $\re \ip{\cdot}{\cdot}_{\Kiltwo}$.
 The symbol
 ${}^{\sigma \perp}$ denotes symplectic complement
 with respect to the symplectic form
 $\im \ip{\cdot}{\cdot}$,
 and ${}^{\re \perp}$ means orthogonality with respect to
 the real inner product
 $\re \ip{\cdot}{\cdot}$.
 The conjugation on both $\Kil$ and $\Kiltwo$ is denoted by $K$,
 and we employ the conjugate-linear operator
 $K^\pi:= K \circ \pi = \pi \circ K$,
 where $\pi$ is the sum-flip on $\Kiltwo$, and
 the real-linear operator
 \begin{equation}
 \label{iota defined}
 \iota := \begin{bmatrix} I \\ -K \end{bmatrix}:
 \Kil \to \Kiltwo, \quad
 f \mapsto \binom{f}{-\ol{f}}.
\end{equation}
 Let $(\Sigmao, \subspace)$ consist of a real subspace $\subspace$
 of $\Kil$ and an operator $\Sigmao$ 0n $\Kil^{\op 2}$ with domain
  $\Lin_\Comp \iota(\subspace)$, and \emph{assume that the following
  hold}:
  \begin{subequations}
  \label{I assump}
  \begin{equation}
 \label{dense assump}
 \subspace \text{ is dense in } \Kil,
  \end{equation}
\begin{equation}
 \label{closable assump}
 \Sigmao \text{ is closable, and}
  \end{equation}
\begin{equation}
 \label{symplectic assump}
 \Sigmao \circ \iota \text{ is symplectic.}
  \end{equation}
 \end{subequations}
  Set $\Sigma:=\ol{\Sigmao}$.
 Note the following, in which $R:= \Ran \iota$:
 \begin{align}
 &\subspace \cap i \subspace \text{ is dense in } \Kil;
 \nonumber
  \\
 &R \cap i R = \{0\} \text{ and }
 R + i R = \Kiltwo;
  \nonumber
 \\
 &\Dom \Sigmao
 \text{ is dense in } \Kiltwo.
  \nonumber
 \end{align}
 Recalling the Fock-Weyl operator notation
 described in the introduction,
 we define
 \begin{align}
 &\Noise_{(\Sigma,\subspace)} :=
 (\Weyl_{\Sigmao})'' \text{ where }
 \Weyl_{\Sigmao}
 := \Lin \{ W(f): f \in \subspace \}
 \text{ and }
 W := W_0 \circ \Sigmao \circ \iota;
 \nonumber
 \\
 &\Omega :=
 \Omega_{\Kiltwo},\
 \Fock := \Fock_{\noisetwo},\
  \text{ and write } V^{(1)}
  \text{ for the natural isometry }
  \Kiltwo \to \Fock;
  \nonumber
 \\
 &H_1 := \ol{\Sigmao \iota(\subspace)} \text { and }
 H_2 := H_1^{\sigma \perp} = i H_1^{\re \perp} = (i H_1)^{\re \perp}.
 \label{H1 defined}
 \end{align}
  Thus $H_1$ and $H_2$ are closed real subspaces of $\Kiltwo$
 and $V^{(1)} V^{(1) *} = P_{\Fock^{(1)}}$,
 where $\bigoplus_{n\geq 0} \Fock^{(n)}$
 is the eigendecomposition for the number operator on $\Fock$.

 The map $w_f\mapsto W(f)$ defines a representation of $CCR(\subspace)$, and
 the vacuum vector induces the quasifree state on $CCR(\subspace)$
 with characteristic function
 $\widehat{\varphi}(f)=e^{-\frac{1}{2}\norm{\Sigmao \iota(f)}^2}$.

 To the above assumptions on $(\Sigmao, \subspace)$
 we add the following:
\begin{subequations}
 \label{Sigma assumptions: basic}
 \begin{equation}
 \label{range assump}
 \Ran \Sigmao \text{ is dense in } \Kiltwo.
 \end{equation}
 \begin{equation}
 \label{generic assump}
 \text{ the pair } (H_1, H_2) \text{ is in generic position.}
 \end{equation}
 \end{subequations}
 Thus $\Omega$ is cyclic and separating for $\Noise_{(\Sigma,\subspace)}$.

 \begin{thm}
 \label{E-O theorem}
 Let $(\Sigmao, \subspace)$ be as above,
 satisfying~\eqref{I assump} and~\eqref{Sigma assumptions: basic}.
 Set
 $s_\Omega := V^{(1) *} S_\Omega V^{(1)}$ and
 $f_\Omega := V^{(1) *} F_\Omega V^{(1)}$.
 Then the following hold.
 \begin{alist}
 \item
 $S_\Omega P_{\Fock^{(1)}} \supset P_{\Fock^{(1)}} S_\Omega$,
 $V^{(1) *} S_\Omega \subset s_\Omega V^{(1) *}$, and
 $s_\Omega \Sigmao = \Sigmao K^\pi$.
 \item
 $s_\Omega$ is closed and densely defined with core $\Ran \Sigmao$.
 Moreover,
 \[
 \Dom s_\Omega^2 \supset H_1 + i H_1
 \ \text{ and } \
 s_\Omega^2 \eta = \eta
 \qquad
 (\eta \in H_1 + i H_1),
 \]
 with $s_\Omega \zeta = - \zeta$ for $\zeta \in H_1$ and
 $s_\Omega ( \zeta) = \zeta$ for $\zeta \in i H_1$.
\end{alist}
 Let $j_\Omega \delta_\Omega^{1/2}$
 be the polar decomposition of $s_\Omega$.
 \begin{itemize}
 \item[(c)]
 $
 j_\Omega H_1 = H_2$,
 $
 j_\Omega = V^{(1) *} J_\Omega V^{(1)}$,
 $\delta^{1/2}_\Omega = V^{(1)*} \Delta^{1/2}_\Omega V^{(1)}$, and
 \[
 J_\Omega W_0(G) J_\Omega = W_0(j_\Omega G)
 \qquad (G \in H_1).
 \]
\end{itemize}
  Set $\Sigmaoprime := j_\Omega \Sigmao (K \op K)$
  and note that $\Sigmaoprime$ is closable and
 \[
 \Dom \Sigmaoprime = (K \op K) \Dom \Sigmao =\Lin_\Comp \iota( K \subspace).
 \]
 Define $\Noise_{(\Sigma', K\subspace)} := (\Weyl_{\Sigmaoprime})''$ where
 $\Weyl_{\Sigmaoprime} := \Lin \{ W'(g): g \in K \subspace \}$
 and
 $W':= W_0 \circ \Sigmaoprime \circ \iota$.

 \begin{itemize}
 \item[(d)]
 $\Sigmaoprime \circ \iota$
 is a symplectic map from $K \subspace$ to $\Kiltwo$.
 \item[(e)]
 $\Noise_{(\Sigma',K\subspace)} = (\Noise_{(\Sigma,\subspace)})'$.
\item[(f)]
$F_\Omega P_{\Fock^{(1)}} \supset P_{\Fock^{(1)}} F_\Omega$,
 $V^{(1) *} F_\Omega \subset f_\Omega V^{(1) *}$
 and
 $f_\Omega \Sigmaoprime = \Sigmaoprime K^\pi$.
 \item[(g)]
 $f_\Omega$ is closed and densely defined with core $\Ran \Sigmaoprime$.
 \item[(h)]
 $f_\Omega = s_\Omega^*$,
 $\Gamma(s_\Omega) = S_\Omega$
 and
 $\Gamma(f_\Omega) = F_\Omega $.
 \end{itemize}
 \end{thm}
\begin{proof}
 (a)
  For $f\in \subspace$, since
  $S_\Omega \ve(t \Sigmao \iota (f))
 =
 \ve(-t \Sigmao \iota (f))$,
  \begin{align*}
 t^{-1}\big( \ve(t \Sigmao \iota (f)) - \Omega  \big)
 &\to
 V^{(1)} \Sigmao \iota(f) =
 P_{\Fock^{(1)}}
 \ve(t \Sigmao \iota (f)), \text{ and }
 \\
 S_\Omega t^{-1}\big( \ve(t \Sigmao \iota (f)) - \Omega  \big)
 &\to
 - V^{(1)} \Sigmao \iota(f) =
 -  P_{\Fock^{(1)}} \ve(t \Sigmao \iota (f)).
  \end{align*}
 Thus
 $P_{\Fock^{(1)}} W(f)\Omega \in \Dom S_\Omega$ and
$S_\Omega P_{\Fock^{(1)}} W(f)\Omega  = P_{\Fock^{(1)}} S_\Omega
W(f)\Omega$.
 Since $\Weyl_\Sigmao \Omega$ is a core for $S_\Omega$, this implies the
 first inclusion. The second inclusion follows, as does the identity
 \begin{equation}
 \label{s Sigma iota}
 s_\Omega \Sigmao \circ \iota = - \Sigmao \circ \iota.
 \end{equation}
 Since $K^\pi \circ \iota = - \iota$,
 the conjugate-linear operators
 $s_\Omega \Sigmao$ and $\Sigmao K^\pi$ agree on
 $\iota(\subspace)$, and therefore coincide.

 (b)
 Since $S_\Omega$ is closed with core $\Weyl_\Sigmao \Omega$,
 (a) and the adjoint-product-inclusion relation~\eqref{adjoint product}
 imply that
  $s_\Omega$ is closed with core
  $V^{(1) *} \Weyl_\Sigmao \Omega = \Ran \Sigmao$,
  which is dense by assumption.
  Now let $\zeta \in H_1$.
  Then $\zeta = \lim \zeta_n$
  for a sequence $(\zeta_n)$ in $\Sigmao \iota (\subspace)$.
  By~\eqref{s Sigma iota},
  $s_\Omega \zeta_n = - \zeta_n \to - \zeta$.
  Since $s_\Omega$ is closed,
  this implies that
  $\zeta \in \Dom s_\Omega$ and $s_\Omega \zeta = - \zeta$.
  Also, by conjugate linearity, $s_\Omega i \zeta = i \zeta$.
  It follows that $H_1 + i H_1 \subset \Dom s_\Omega^2$ and
  $s_\Omega^2 \eta = \eta$ for $\eta \in H_1 + i H_1$.
  This proves (b).

 (c)
 This is proved in~\cite{osterwalder} using Halmos'
 two subspaces paper (\cite{halmos}); see also Chapter 7 of~\cite{petz}.

 (d)
 $\Sigmaoprime \circ \iota$ is symplectic since,
 for $f,g\in \subspace$,
 \begin{align*}
 \im \ip{\Sigmaoprime \iota(\ol{f})}{\Sigmaoprime \iota(\ol{g})}
 &=
 \im \ip{ j_\Omega \Sigmao \iota(f)}{j_\Omega \Sigmao \iota(g)}
 \\
 &=
 - \im \ip{ \Sigmao \iota(f)}{ \Sigmao \iota(g)}
 =
 - \im \ip{f}{g} =
 \im \ip{\ol{f}}{\ol{g}}.
 \end{align*}
 Since $(K \op K) \circ \iota = \iota \circ K$ and $j_\Omega$ is isometric,
 the density of $\Ran \Sigmaoprime$ follows from (c):
 \[
 \ol{\Ran} \Sigmaoprime = j_\Omega \ol{\Ran} \Sigmao =
 j_\Omega H_1 = H_2.
 \]

 (e)
 By (c),
\begin{align*}
 W'(\ol{f})
 = W_0(\Sigmaoprime \iota(\ol{f}))
 &= W_0(j_\Omega \Sigmao \iota(f))
 \\
 &= J_\Omega W_0(\Sigmao \iota(f)) J_\Omega
 = J_\Omega W(f) J_\Omega
\end{align*}
so, by Tomita's Theorem,
\[
 \Noise_{(\Sigma',K\subspace)} =
 (\Weyl_{\Sigmaoprime})'' =
 (J_\Omega \Weyl_{\Sigmao} J_\Omega)'' =
 J_\Omega (\Weyl_{\Sigmao})'' J_\Omega =
 J_\Omega \Noise_{(\Sigma,\subspace)} J_\Omega =
 (\Noise_{(\Sigma,\subspace)})'.
\]

 (f)\&(g)
 By the assumptions on $(\subspace, \Sigmao, H_1, H_2)$,
 and what has been already proved,
 the pair $(K\subspace, \Sigmaoprime)$ consists of
 a dense real subspace of $\Kil$ and
 a closable operator
 satisfying~\eqref{Sigma assumptions: basic},
 with $(H_2, H_1)$ in place of $(H_1, H_2)$.
 Since, by (e),
 the $S$-operator for $(N_{\Sigma'},\Omega)$
 is $F_\Omega$, (f) and (g) are precisely what
 results from applying (a) and (b) to the pair
 $(K\subspace, \Sigmaoprime)$.

 (h)
 The identity $s_\Omega^* = f_\Omega$
 follows from
 (a) and Part (c) of Lemma~\ref{Lemma 1}.
 For $f\in \subspace$,
 \[
 \Gamma( s_\Omega) \ve(\Sigmao \iota(f) ) =
 \ve( s_\Omega \Sigmao \iota(f) ) =
 \ve( - \Sigmao \iota(f) ) =
  S_\Omega \ve(\Sigmao \iota(f)).
 \]
 The closed operators $S-\Omega$ and $\Gamma(s_\Omega)$
 therefore agree on
 $\Exps\big(\Sigmao \iota (\subspace)\big) = \Weyl_{\Sigmao}\Omega$,
 which is a core for $S_\Omega$,
 so $S_\Omega \subset \Gamma(s_\Omega)$.
 Applying this with $(\Sigmao, S_\Omega)$ replaced by
 $(\Sigmaoprime, F_\Omega)$ gives $F_\Omega \subset \Gamma(f_\Omega)$,
 so we also have
 \[
 S_\Omega =
 F_\Omega^* \supset \Gamma(f_\Omega)^* =
 \Gamma(f_\Omega^*) =
 \Gamma(s_\Omega).
 \]
 Therefore the required equality holds, and the proof is complete.
\end{proof}

 We make two simple observations,
 as motivation for the following result.

 \begin{rems}
 If $\Sigmao$ is closed (so that $\Sigma = \Sigmao$), then
 \[
 \Ran \Sigma \subset H_1 + i H_1.
 \]
 Thus, if $\Sigmao$ is surjective (and thus also closed) then
 \begin{equation}
 \label{H = K}
 H_1 + i H_1 = \Kil^{\op 2}.
 \end{equation}
 \end{rems}

 \begin{propn}
 \label{cor to E-O}
 Let $(\Sigmao, \subspace)$ be as in Theorem~\ref{E-O theorem},
 and assume~\eqref{H = K}.
 Then the following hold\tu{:}
  \begin{alist}
 \item
 $s_\Omega^2 = I_{\Kil^{\op 2}}$,
 in particular $s_\Omega$ is bounded\tu{;} it is given by
 \[
 s_\Omega (\zeta + i \eta) = - \zeta + i \eta
 \qquad
 (\zeta, \eta \in H_1).
 \]
 \item
 If also $\Sigma$ is surjective then
 \begin{rlist}
 \item
 $s_\Omega \Sigma = \Sigma K^\pi$, so
 $s_\Omega = \Sigma K^\pi \Sigma^{-1}$.
 \item
 $(s_\Omega \ot S_\Omega)(\Sigma \ot I_\Fock)
 \subset
 s_\Omega \Sigma \ot S_\Omega
 \subset
 (\Sigma \ot I_\Fock) (K^\pi \ot S_\Omega),
 $
 moreover, the second operator is the closure of the first.
 \end{rlist}
 \item
 If $\Sigma$ is surjective and we assume further that
 there is a real subspace $\domain$ of
 $\subspace$ such that
 \begin{equation}
 \label{refinement}
 \iota( K \domain) \subset \Dom \Sigma^* \Sigma'
 \ \text{ and } \
 \Lin_\Comp \Sigma' \iota (K\domain)
 \text{ is dense in } \Kil^{\op 2},
 \end{equation}
 then the conclusion in \tu{(}b\tu{)}\tu{(}ii\tu{)} has the following refinement\tu{:}
 \[
 \Dom (s_\Omega \ot S_\Omega) ( \Sigma \ot I_\Fock) =
 \Dom s_\Omega \Sigma \ot S_\Omega \cap \Dom \Sigma \ot I_\Fock.
 \]
 \end{alist}
 \end{propn}

 \begin{proof}
 (a)
 This follows immediately from Part (b) of
 Theorem~\ref{E-O theorem}.

 (b) (i)
 We have $s_\Omega \Sigmao = \Sigmao K^\pi$ and so,
 since $s_\Omega^2 = I_{\Kil^{\op 2}}$,
 $\Sigmao = s_\Omega \Sigmao K^\pi$.
 Since also ($K^\pi)^2 = I_{\Kil^{\op 2}}$, it follows that
 $\Sigma = \ol{s_\Omega \Sigmao K^\pi} = s_\Omega \Sigma K^\pi$
 and (i) follows.

 (b) (ii)
 Since $s_\Omega \Sigma$ is closed and $\Sigma^{-1}$ is bounded
 we have
 \[
 s_\Omega \ot S_\Omega =
 s_\Omega \Sigma \Sigma^{-1} \ot S_\Omega =
 \big( s_\Omega \Sigma \ot S_\Omega \big) ( \Sigma^{-1} \ot I_\Fock)
  =
 \big( \Sigma K^\pi \ot S_\Omega \big) ( \Sigma^{-1} \ot I_\Fock )
 \]
 (by Part (d) of Proposition~\ref{Lemma 3}), therefore
 \begin{align*}
 ( s_\Omega \ot S_\Omega ) (\Sigma \ot I_\Fock)
 \subset
 s_\Omega \Sigma  \ot S_\Omega
 &=
 \Sigma K^\pi \ot S_\Omega
 \\
 &\subset
 ( \Sigma \ot I_\Fock ) ( K^\pi \ot S_\Omega ),
 \end{align*}
 by Part (e) of Proposition~\ref{Lemma 3}.
 Since
 $s_\Omega \Sigma  \ot S_\Omega$ is closed and the domain of the LHS
 of this inclusion contains $\Dom \Sigma \otul \Dom S_\Omega$
 which is a core for the middle term, (ii) follows.

 (c)
 Let $x \in
 \Dom s_\Omega \Sigma \ot S_\Omega \cap \Dom \Sigma \ot I_\Fock$.
 Since
 $(s_\Omega \Sigma \ot S_\Omega)^* =
 \Sigma^* f_\Omega \ot F_\Omega$,
 to see that
 $x \in
 \Dom (s_\Omega \ot S_\Omega) ( \Sigma \ot I_\Fock)$
 if suffices to verify that
 \begin{equation}
 \label{verify}
 \big\langle
 (f_\Omega \ot F_\Omega) \alpha,
 (\Sigma \ot I_\Fock) x
 \big\rangle
 =
 \big\langle
 (\Sigma^* f_\Omega \ot F_\Omega) \alpha,
  x
 \big\rangle
 \end{equation}
 for all vectors $\alpha$ from
 a subset of $\Dom \Sigma^* f_\Omega \ot F_\Omega$
 which is a core for $f_\Omega \ot F_\Omega$.
 Since $f_\Omega$ is bounded, it suffices to verify~\eqref{verify}
 for vectors $\alpha$
 of the form $u \ot T\Omega$ where $T\in \Noise'_{(\Sigma,\subspace)}$ and
 $u$
 is
 from a total subset of $\Kil^{\op 2}$.
 By assumption we may take $u$ from $\Sigma' \iota (K \domain)$.
 Now
 \[
 (f_\Omega \ot F_\Omega) \Sigma' \iota(\ol{g}) \ot T\Omega =
 f_\Omega \Sigma' \iota(\ol{g}) \ot T^* \Omega =
 - \Sigma' \iota(\ol{g}) \ot T^* \Omega
 \]
 for all $g\in\domain$ and $T\in\Noise'_{(\Sigma,\subspace)}$ and
 so, for such $\alpha$,
 \[
 \text{LHS of } \eqref{verify} =
 \big\langle
 - \Sigma^* \Sigma' \iota(\ol{g}) \ot T^*\Omega,
 x
 \big\rangle
 =
 \text{ RHS of } \eqref{verify},
 \]
 as required.
 \end{proof}

 The elementary observation contained in the following lemma
 is relevant to the examples below.
\begin{lemma}
\label{V plus KV}
 For any real subspace $V$ of $\Kil$,
 \[
 V \op \{ 0 \} =
 \big\{
 \iota(f) - i \iota(if) : f \in V
 \big\}
 \ \text{ and } \
 \{ 0 \} \op K V =
 \big\{
 \iota(f) + i \iota(if) : f \in V
 \big\}.
 \]
 In particular,
 if $V$ is a complex subspace of $\Kil$ then
 \[
 \Lin_\Comp \iota(V) =
 V \op K V.
 \]
\end{lemma}
 \begin{proof}
 Let $J$ be the real-linear map $f \mapsto if$ on $\Kil$.
 Then
 \[
 \iota = \begin{bmatrix} I \\ -K \end{bmatrix}
 \ \text{ and }\
(J \op J) \iota J = - \begin{bmatrix} I \\ K \end{bmatrix},
 \]
 so
 \[
 \iota - (J \op J) \iota J = 2
 \begin{bmatrix} I \\ 0 \end{bmatrix}
  \ \text{ and }\
  \iota + (J \op J) \iota J = - 2
 \begin{bmatrix} 0 \\ K \end{bmatrix}.
 \]
 The result follows.
 \end{proof}

\begin{example}[Guage-invariant quasifree states]
 Let $\subspace$ be the complex subspace $\Dom T^{1/2}$ of $\Kil$,
 where $T$ is a nonnegative selfadjoint operator on $\Kil$,
  and let $ \Sigmao$ be the nonnegative selfadjoint operator
\[
 \Sigma_T :=
  \begin{bmatrix}
 \sqrt{I + T} &
 \\
 & K\, \sqrt{T}\, K
 \end{bmatrix}.
\]
 It follows from the functional calculus for $T$ that
 $\Sigmao \circ \iota$ is symplectic and
 $\norm{\Sigmao \iota(v)} =
 \norm{\sqrt{I + 2T} v} \geq \norm{v}$ ($v\in\subspace$),
 so there is a unique quasifree state on $CCR(\subspace)$ with
 characteristic function
 \[
 \wh{\varphi}_T : v \mapsto
 e^{-\frac{1}{2} \norm{\sqrt{I + 2T}  v}^2}.
 \]
 Moreover, since
 $\wh{\varphi}_T(z v) = \wh{\varphi}_T(v)$ ($z\in\Torus$),
 the state is guage-invariant.
 Note also that
 $H_1$ and $H_2$ are the closures of the ranges of the
 respective operators
\[
 \begin{bmatrix}
 \sqrt{I+T} \\ -K \sqrt{T}
 \end{bmatrix}
 \text{ and }
 \begin{bmatrix}
 \sqrt{T} \\ -K \sqrt{I+T}
 \end{bmatrix}.
\]
 The degenerate case where $T = 0$ is the Fock state.
 On the other hand if $T$ is injective then
 $\Sigma_T$ has dense range and
 it is straightforward to verify that
 $(H_1,H_2)$ is in generic
 position, so Theorem~\ref{E-O theorem} applies.
 The associated operators are then
 \[
 j_\Omega =
 \begin{bmatrix}
  & K
  \\
  K &
 \end{bmatrix},\
\delta_\Omega^{1/2} =
 \begin{bmatrix}
 \sqrt{I+T^{-1}}^{-1} &
 \\
 & K \sqrt{I+T^{-1}} \ K
 \end{bmatrix}\
 \text{ and }\
 \Sigmaoprime=\Sigma_T',
 \]
 where
 \[
 \Sigma'_T :=
 \begin{bmatrix}
 & \sqrt{T}
 \\
 K \sqrt{I+T} \ K &
 \end{bmatrix}.
 \]
 Thus $j_\Omega = K^\pi$.
 Note that $\Sigma_T$ and $\Sigma_T'$
 are both closed, and~\eqref{refinement} holds with
 $\domain$ equal to $\Dom T$ since
 \[
 \Sigma^*_T \Sigma_T' =
 \frac{1}{2}
 \begin{bmatrix}
 & \sqrt{T(I+T)} \\ K \sqrt{T(I+T)} \, K
 \end{bmatrix}.
 \]
 Thus,
 if $T$ is bijective
 then so is $\Sigma_T$ and
 Proposition~\ref{cor to E-O} applies.
 Note that in this case
 $T^{-1}$ is bounded so
 the boundedness of
 $\delta^{1/2}_\Omega$, and thus also of
 $s_\Omega$, is manifest.
 Moreover, setting $A = \log(I+T^{-1})$,
 we have $I+2T = \coth A$.
 The case $A = \frac{\beta \hbar}{2} I$ then corresponds to
 the temperature state of $CCR(\Kil)$
 with inverse temperature $\beta$ (\cite{BrR}).
\end{example}

\begin{example}[Squeezed states]
 The above guage-invariant quasifree states
 may be `squeezed' by composing with the
 Bogoliubov automorphism $\phi_Q$ of $CCR(\subspace)$
 induced by a symplectic automorphism $Q$ of $\subspace$.
 We use the following structure theorem from
 \cite{hon}.
 If either $\Kil$ is separable,
 or $Q$ is bounded
 (as a densely defined operator on $\Kil$,
 viewed as a real Hilbert space),
 then $Q$
 is the restriction of an operator of the form
 \[
 U( \cosh\! P - K' \sinh\! P)
 \]
 to $\subspace$, where $U$, $K'$ and $P$ are operators on $\Kil$,
 $U$ being unitary,
 $K'$ another conjugation, and
 $P$ a second nonnegative selfadjoint operator, and
 the following consistency conditions hold:
  \begin{alist}
  \item
 For
 $R \in
 \big\{
 U \cosh\! P,\ U K' \sinh\! P,\ \cosh\! P \ U^*,\ \sinh\! P \ K' U
 \big\}$,
 \[
 \subspace \subset \Dom R \text{ and }
 R(\subspace) \subset \subspace.
 \]
  \item
  $K'$ commutes with the spectral projectors of $P$.
 \item
 $\subspace$ is a core for $\sinh^2\! P$;
 \end{alist}
 moreover if $(\wt{U}, \wt{K}', \wt{P})$
 is another such parameterisation of $Q$ then
 $(\wt{U}, \wt{P}) = (U,P)$, and
 $\wt{K}'$ and $K'$ agree on $\ol{\Ran} P$.
  In terms of these,
  $\Sigma_T \circ \iota \circ Q = \Sigma_{T,Q} \circ \iota$
  where
 \[
 \Sigma_{T,Q} = \Sigma_T ( U \op K U K' ) \Gamma (I \op K' K)
 \text{ for }
 \Gamma =
 \begin{bmatrix}
 \cosh\! P & \sinh\! P
 \\
 \sinh\! P
 & \cosh\! P
 \end{bmatrix},
 \]
 the corresponding quasifree
 state on $CCR(\subspace)$ has characteristic function
 \[
 \wh{\varphi}_{T, Q}: v \mapsto
 e^{-\frac{1}{2} \norm{\sqrt{I + 2T}\ Q v}^2}.
 \]
 If $\Sigmao := \Sigma_{T,Q}$ is closable with dense range
 (for example if $P$ is bounded)
 then Theorem~\ref{E-O theorem} applies,
 $H_1$, $H_2$, $j_\Omega$ and $\delta_\Omega^{1/2}$ are as
 in the gauge-invariant case above, and
 \[
  \Sigma'_{T,Q} :=
 \Sigma'_T (K \op K) ( U \op K U K' ) \Gamma ( K \op K' ).
 \]
\end{example}

\section{Quasifree states for stochastic analysis}
 \label{section: qf states for stochastic analysis}

 We now specialise our quasifree states for stochastic analysis, and
 we identify natural conditions on a pair
 $(\Sigmao, \subspace)$ ---
 consisting of a dense real subspace $\subspace$ of $\Kil$
 and closable operator $\Sigmao$ on $\Kiltwo$ with domain
 $\Lin_\Comp \iota(\subspace)$ ---
 for
 Assumptions~\eqref{I assump} and~\eqref{Sigma assumptions: basic}
 to hold,
 so that
 Theorem~\ref{E-O theorem} applies. We then show that this entails
 a key commutation relation between It\^o integration and the
 Tomita-Takesaki operators.

 The notation is as for the previous section, but now
 $\Kil = L^2(\Rplus;\noise)$
 as in Section~\ref{section: Ito commutation relations}
 except that now
  $\noise$ is the complexification
 of a separable real Hilbert space $\noise^\Real$.
 Thus $\Kiltwo = L^2(\Rplus;\noisetwo)$ and
 $\Kil$ is the complexification of
 $L^2(\Rplus;\noise^\Real)$;
 the conjugation on $\Kil$ being that induced by the conjugation
 on $\noise$ pointwise:
 \[
 \ol{f}(t):= \ol{f(t)}
 \qquad
 (t\in\Rplus).
 \]
 \textbf{Assumptions.}
 Setting $\Sigma := \ol{\Sigmao}$ and
 $\Sigma_t := V_t^* \Sigma V_t$
 where
 $V_t$ is the inclusion map $\Kiltwo_t  \to \Kiltwo$,
 we now make the following assumptions on
 the pair $(\Sigmao, \subspace)$:
 \begin{alist}
 \item
 $\Sigmao \circ \iota$ is symplectic and, for all $t\in\Rplus$,
 \item
 $ \subspace_t := p_t(\subspace) \subset \subspace$,
 \item
 $p_t \Sigmao \subset \Sigmao p_t$,
 \item[(d)]
 $\Sigma_t$ is bijective with bounded inverse,
 \end{alist}
 and consider the further alternative assumptions:
 \begin{alist}
 \item[(e)] $V_t^*(H_1 + i H_1) = \Kil^{\op 2}_t$ and  there is a real subspace $\domain_t$ of
 $W_t^*\subspace_t$, where $W_t$ is the inclusion $\Kil_t\to\Kil$, such that \\
 ${}\hspace{0.75cm}\iota( K_t \domain_t) \subset \Dom \Sigma^*_t \Sigma'_t
   \ \text{ and } \
   \Lin_\Comp \Sigma'_t \iota (K_t\domain_t)
   \text{ is dense in } \Kil_t^{\op 2}. $
 \item[($\text{e}_+$)] $\Sigma_t$ is bounded for all $t\in\Real_+$.
 \end{alist}

\begin{rems} (i) Here are some consequences of Assumptions (a)--(d).

 ($\alpha$)
 $\Sigma\, \affiliated  L^\infty(\Rplus) \otol B(\noisetwo)$;
 this follows from Lemma~\ref{invariance = affiliation}.

 ($\beta$)
  For all $t\in\Rplus$,
  $\Sigma_t$ is closed with core $\Lin_\Comp \iota(\subspace_t)$;
  this follows from Part (c) of Lemma~\ref{Lemma 1}.

 ($\gamma$)
 $\Sigma$ is injective.

 ($\delta$)
 $\bigcup_{t\geq 0} \Lin_\Comp \iota(\subspace_t)$ is
 a core for $\Sigma$.

 ($\epsilon$)
 For all $t\in\Rplus$,
 $\Ran \Sigmao_t$ is dense in $\Kiltwo_t$,
 where $\Sigmao_t:= V_t^* \Sigmao V_t$.

 ($\zeta$)
 For all $t\in\Rplus$, $p_t\Sigma\subset \Sigma p_t$; this follows from Lemma \ref{Lemma 1} (a).

 \noindent (ii)
 Notice that (e) is a localised version of the hypotheses in Proposition \ref{cor to E-O}.
 Indeed (e) implies the local boundedness property $V_t^*s_\Omega V_t\in B(\Kiltwo_t)$
 for every ${t\in\Rplus}$ as follows.
 Setting $\subspace^t := (I - p_t) \subspace$, the
  assumptions (a)-(d) give us a decomposition
  $\Sigma=\Sigma_t\oplus\Sigma^t$ on
 $\subspace=\subspace_t\oplus\subspace^t$, and thus a pair of von Neumann algebras
 $\Noise_{(\Sigma_t,\subspace_t)}$ and $\Noise_{(\Sigma^t,\subspace^t)}$ for each $t\geq0$.
 Using Weyl operators one sees that
 $\Noise_{(\Sigma,\subspace)}=
 \Noise_{(\Sigma_t,\subspace_t)}\otol\Noise_{(\Sigma^t,\subspace^t)}$.
 By Theorem \ref{E-O theorem} this gives the decomposition
 $$
 \Gamma(s_\Omega)=S_\Omega=S_{\Omega_{t)}}\otimes S_{\Omega_{[t}}=
 \Gamma(s_{\Omega_{t)}})\otimes\Gamma(s_{\Omega_{[t}})=
 \Gamma(s_{\Omega_{t)}}\oplus s_{\Omega_{[t}}),
 $$
 so that $s_\Omega=s_{\Omega_{t)}}\oplus s_{\Omega_{[t}}$.
  It follows from Proposition \ref{cor to E-O} (a) that,
  for all $t\geq 0$,
 $V_t^*s_\Omega V_t=s_{\Omega_{t)}}\in B(\Kiltwo_t)$.

 \noindent (iii)
 Assumption ($\text{e}_+$) implies (e).
 To see this note that if ($\text{e}_+$) holds then $\Sigma_t$ is bounded and invertible,
 and $V_t^*(H_1+iH_1)$ is a closed subspace of $\Kiltwo_t$
 containing $\Ran \Sigmao_t$, and so equals $\Kiltwo_t$.
 Now we claim that $W_t^*\subspace_t$ itself satisfies the conditions required from (e).
 To see this, note that
 $$
 \Sigma'_t\supset V_t^*j_\Omega \Sigmao (K\oplus K) V_t
 $$
 is defined on all of $\iota(K_tW_t^*\subspace_t)$ and, since $\Sigma_t$ is bounded,
 $\Sigma_t^*\Sigma'_t$ is defined here too.
 The decomposition $s_\Omega=s_{\Omega_{t)}}\oplus s_{\Omega_{[t}}$ in Remark (ii)
 gives $j_\Omega=j_{\Omega_{t)}}\oplus j_{\Omega_{[t}}$. Thus
 \begin{align*}
 \Sigma'_t\iota(K_tW_t^*\subspace_t)
 &=
 V_t^*j_\Omega \Sigmao (K\oplus K) V_t\iota(K_tW_t^*\subspace_t)
 \\&=
 V_t^*j_\Omega \Sigmao V_t\iota(W_t^*\subspace_t)
 =
 j_{\Omega_{t)}}\Sigma_t\iota(W_t^*\subspace_t).
 \end{align*}
 It now follows from $(\beta)$ and $(\epsilon)$ that
 $\Sigma'_t\iota(K_t W_t^* \subspace_t)$ is dense in $\Kiltwo_t$.

\end{rems}

 Recall that
 \begin{equation*}
 H_1:= \ol{\Sigmao \iota(V)} \text{ and }
 H_2 := H_1^{\sigma \perp} = i H_1^{\re \perp} = (i H_1)^{\re
 \perp}.
 \end{equation*}

 \begin{thm}
 \label{S Omega Ito}
 Under Assumptions \tu{(a)--(d)}
 on the pair $(\Sigmao, \subspace)$,
 the following hold.
 \begin{alist}
 \item
 $(H_1, H_2)$ is in generic position,
 so Theorem~\ref{E-O theorem} applies.
 \item
 $s_\Omega$ and $f_\Omega$ are affiliated to
 $L^\infty(\Rplus)\otol B(\noisetwo)$.
 \item
 $
 S_\Omega \Ito = \Ito \circ \big( s_\Omega \ot_\Omega S_\Omega \big)
 $.
 \end{alist}
 \end{thm}
 \begin{proof}
 (a)
 Let us abbreviate ${}^{\re \perp}$ to ${}^{\perp}$.
 We first make a general observation about elements of $H_1$.
 For $F\in H_1$, let $(f^n)$ be a sequence in $\subspace$ such that
 $\Sigma \iota(f^n) \to F$ and let $t\geq 0$.
 Then
 \begin{equation}
 \label{extra label}
 \Sigma_t^{-1} V_t^* F =
 \lim \Sigma_t^{-1} V_t^* \Sigma \iota(f^n) =
 \lim V_t^* \iota(f^n).
 \end{equation}
 Thus $(f^n_{[0,t]})$ converges, to $f_t \in \ol{\subspace_t}$ say, where
 $V_t^* \iota(f_t) = \Sigma_t^{-1} V_t^* F$ so
 $\iota(f_t) \in \Dom \Sigma$
 and,
 by~\eqref{extra label},
 $V_t^* F = \Sigma_t V_t^* \iota(f_t) = V_t^* \Sigma \iota(f_t)$, so
 $\Sigma \iota(f_t) = F_{[0,t]}$.

 (i)
  Let $F \in H_1 \cap H_2 = H_1 \cap (i H_1)^\perp$.
  Then, for all $t\in\Rplus$ and $g\in
  \subspace$,
  \begin{align*}
 0 =
 - \re \ip{F}{i \Sigma \iota(g_{[0,t]})}
 &= \im \ip{F}{\Sigma \iota(g_{[0,t]})}
 \\
 &=
 \im \ip{F_{[0,t[}}{\Sigma \iota(g)}
 =
 \im \ip{\Sigma \iota(f_t)}{\Sigma \iota(g)}
 =
 \im \ip{f_t}{g},
  \end{align*}
  since $\Sigma \circ \iota$ is symplectic.
 Thus, for $h \in \subspace \cap i \subspace$,
 $\im \ip{f_t}{h} = 0$ and
 $\re \ip{f_t}{h} = \im \ip{f_t}{ih} = 0$.
 Since $\subspace \cap i \subspace$ is dense in $\Kil$,
 this implies that $f_t = 0$ so $F_{[0,t]} = 0$.
 Letting $t$ vary we see that $F = 0$.
 Thus $H_1 \cap H_2$ is trivial.

(ii)
 By Remark ($\gamma$) it follows that $\Ran \Sigmao$ is dense.
 Therefore the triviality of $H_1^\perp \cap H_2$ follows from the
 relation
  \[
   H_1^\perp \cap H_2 =
 H_1^\perp \cap (i H_1)^{\perp} =
 ( H_1 + i H_2)^{\perp} \subset
 \big( \Ran \Sigmao \big)^{\perp} = \{0\}.
\]

(iii)
 Let $F\in H_1 \cap H_2^{\perp} = H_1 \cap i H_1$
 and set $R = \Ran \iota$.
 Then, for each $t\geq 0$,
 $F_{[0,t]} = \Sigma \iota(f_t)$ and
 $i F_{[0,t]} = \Sigma \iota(g_t)$
 for some $f, g \in \Kil$.
 Therefore, for each $t\geq 0$,
 \[
 F_{[0,t]} \in
 \Ran \Sigma \circ \iota \cap i \Ran \Sigma \circ \iota =
 \Sigma( R \cap i R) = \{ 0 \},
 \]
 so $F = 0$.
 Thus $H_1 \cap H_2^{\perp}$ is trivial.

  (iv)
 In view of the identity
 $H_1^{\perp} \cap H_2^{\perp} =
 iH_2 \cap iH_1 =
 i(H_1 \cap H_2)$,
 (i) implies that this subspace is trivial too.
 Therefore (a) holds.

(b)
 Since $f_\Omega = s_\Omega^*$, it
 suffices to show that $s_\Omega$ is so affiliated.
 Let $t\geq 0$. Then, for $f\in \subspace$,
 \begin{align*}
 p_t \Sigmao \iota(f) &=
 \Sigmao \iota(f_{0,t[})
 \in V^{(1) *} \Weyl_\Sigmao \Omega,
 \text{ and }
 \\
 s_\Omega p_t \Sigmao \iota(f) &=
 - \Sigmao \iota(f_{0,t[}) =
 p_t s_\Omega \Sigmao \iota(f).
 \end{align*}
 Thus $(p_t s_\Omega)_{| V^{(1) *}  \Weyl_\Sigmao \Omega} \subset s_\Omega p_t$.
 Since $s_\Omega p_t$ is closed and, by Part (b) of Theorem~\ref{E-O theorem},
 $V^{(1) *} \Weyl_\Sigmao \Omega$ is a core for $s_\Omega$,
 this implies that
 $p_t s_\Omega \subset s_\Omega p_t$.
 (b) therefore follows from
 Lemma~\ref{invariance = affiliation}.

 (c)
 This now follows from
 Theorem~\ref{Ito second quantisation commutation}.
 \end{proof}

\begin{rem}
 In~\cite{squeeze}
 an abstract noncommutative stochastic calculus
 is related to squeezed states,
 additive cocycles
 with respect to the natural shift are considered, and
 an It\^o table derived.
 In~\cite{LM} we derive the It\^o table for the
 general quasifree setting considered here.
\end{rem}

\begin{examples}
 For the squeezed quasifree states discussed in
 Section~\ref{section: CCR algebras and qf states},
 Assumptions (a)--(d) are satisfied if
 $T$ and $P$ are affiliated to $L^\infty(\Rplus) \otol B(\noise)$,
 $P$ is locally bounded,
 $K'$ is a pointwise conjugation on $\Kil$:
 $(K'f)(t) = k' f(t)$ ($t\in \Rplus$)
 for some conjugation $k'$ on $\noise$, and
 there is $\alpha \in L^\infty_{\loc}(\Rplus)$ such that
 $\alpha > 0$ almost everywhere and
 $T \geq M_{\alpha^{-1}} \ot I_\noise$.
 Assumption  ($\text{e}_+$)  is satisfied too if
 $\alpha$ may be chosen so that also
 $T \leq M_{\alpha} \ot I_\noise$.

 On the other hand,
 if $P$ is bounded and $T=I_{L^2(\Rplus)}\otimes Q$ where
 $Q$ is a closed, densely defined, unbounded and bijective operator
 then the resulting pairs ($\Sigmao, \subspace$) satisfy
 (a)--(e), but not ($\text{e}_+$).
\end{examples}

\section{Modified It\^o integral}
\label{section: modified Ito}

 In this section we establish the appropriate analogue of the
 abstract Kunita-Watanabe Theorem
 at the vector process level.

 Let $(\Sigmao, \subspace)$ be as in
 Section~\ref{section: qf states for stochastic analysis},
 take the notations
 $\Sigma$, $\Omega$, $\Noise_{(\Sigma,\subspace)}$ and  $V_t$ from
 Sections~\ref{section: qf states for stochastic analysis}
 and~\ref{section: CCR algebras and qf states},
 and
 fix a von Neumann algebra $\Init$
 acting on a separable Hilbert space $\init$,
 which we refer to as the \emph{initial algebra},
 with cyclic and separating vector $\cyclic$.
 Assumptions (a)-(e) are in operation.
 and we set $\Noise = \Noise_{(\Sigma, \subspace)}$,
 \begin{equation}
 \label{notations M to H}
 \InitNoise = \Init \otol \Noise, \
 \xi = \cyclic \ot \Omega, \
  \Xi = \InitNoise' \xi, \
 S = S_\xi,
 \text{ and }\,
 \initFock = \init \ot \Fock_{\noisetwo}.
 \end{equation}
 Thus the vector $\xi$ is cyclic and separating for
 the von Neumann algebra $\Mil$,
 $S = S_\cyclic \ot S_\Omega$
 (\cite{StZ}, 10.7),
 and
 the Hilbert space $\initFock$ is separable.
 Also write
 $P^\Omega_t$ for $V_\Omega^* P_t V_\Omega$,
 the restriction of
 $P_t$ on $\Kiltwo \ot \initFock$ to
 the subspace
 $L^2_{\Omega}(\Rplus; \noisetwo \ot \initFock)$ ---
 which conveniently extends to a map
 $L^2_{\Omega, \loc}(\Rplus; \noisetwo \ot \initFock)
 \to
 L^2_{\Omega}(\Rplus; \noisetwo \ot \initFock)$
 in a natural way.
 Finally, set
 \[
 k^\pi:= k \circ \pi = \pi \circ k, \
 K^\pi:= K \circ \pi \text{ and }
 K_t^\pi:= K_t \circ \pi,
 \]
 where $k$, $K$ and $K_t$ are the conjugations on
 $\noisetwo$, $\Kiltwo$ and $\Kiltwo_t$ respectively,
 and $\pi$ is the sum-flip on each of these orthogonal sums.

\begin{lemma}
 The following holds\tu{:}
 $$
 (s_\Omega\otimes S)(\Sigma V_t\otimes I_\initFock) \subset
 (\Sigma V_t\otimes I_\initFock)(K^\pi_t\otimes S).
 $$
 Under Assumption \tu{(}$\text{e}_+$\tu{)} this can be strengthened to an equality.
\end{lemma}

 \begin{proof}
 Recall that $s_\Omega=s_{\Omega_{t)}}\oplus s_{\Omega_{[t}}$
 and, by Proposition \ref{cor to E-O} (a),
 $s_{\Omega_{t)}}$ is bounded, hence $s_{\Omega_{t)}}\Sigma_t=\Sigma_tK^\pi_t$.
 Now, since $V_t$ is an isometry, applying Lemma \ref{Lemma 1} (d)
 together with Proposition \ref{cor to E-O} (b) we get
 \begin{align}
 (s_\Omega\otimes S)(\Sigma V_t\otimes I_\initFock)
 &=
 (s_\Omega\otimes S)(P_t\otimes I)(\Sigma V_t\otimes I_\initFock)
 \nonumber
 \\
 &=(V_t\otimes I)(s_{\Omega_t]}\otimes S)(\Sigma_t\otimes I_\initFock)
 \nonumber \\
 & \subset (V_t\Sigma_t\otimes I_\initFock)(K^\pi_t\otimes S)=
 (\Sigma V_t\otimes I_\initFock)(K^\pi_t\otimes S).
 \label{fundamentalinclusion}
 \end{align}

 If Assumption ($\text{e}_+$) holds then $( s_\Omega \ot S )(\Sigma V_t \ot I_\initFock)$
 is closed and
 defined on $\Dom \Sigma V_t \otul \Dom S$,
 which is a core for $s_\Omega\Sigma V_t\ot S$.
 Thus $( s_\Omega \ot S )(\Sigma V_t \ot I_\initFock)\supset s_\Omega\Sigma V_t\ot S=
 \Sigma V_tK^\pi_t\otimes S$ and~\eqref{fundamentalinclusion} can be strengthened to an equality.

\end{proof}

 \begin{lemma}
 \label{tedious}
 The following holds
 \begin{equation}
 \label{L}
 ( s_\Omega \ot_\Omega S )
 (\Sigma \ot_\Omega I_\initFock)
 P_t^\Omega \subset
 ( \Sigma \ot_\Omega I_\initFock )
 ( K^\pi \ot_\Omega S )
 P_t^\Omega
 \qquad
 (t \in \Rplus).
 \end{equation}
 Moreover, under Assumption \tu{(}$\text{e}_+$\tu{)}, this is an
 equality.
 \end{lemma}

\begin{proof}
 Let $t\in\Rplus$. In view of the identity
 \[
 (V_t \ot I_\initFock)
 \big( V_t^* K^\pi V_t \ot S \big)
 ( V_t^* \ot I_\initFock ) =
 ( K^\pi \ot S )
 ( V_t V_t^* \ot I_\initFock )
 \]
 (which follows from Part (d) of Proposition~\ref{Lemma 3}),
 applying Lemma~\ref{lemma: tensor Omega},
 first with $T = \Sigma$ and $X = I_\initFock$
 and last with $T = K$ and $X = S$,
 and Part (b) of Proposition~\ref{Lemma 3} again, we have
\begin{align*}
 \text{LHS of }\eqref{L}
 &=
 V_\Omega^*( s_\Omega \ot S )
 V_\Omega V_\Omega^*
 (\Sigma \ot I_\initFock)
 V_\Omega P_t^\Omega
 \\
 &=
  V_\Omega^*( s_\Omega \ot S )
 (\Sigma \ot I_\initFock)
 V_\Omega P_t^\Omega
 \\
 &=
 V_\Omega^*( s_\Omega \ot S )
 (\Sigma \ot I_\initFock)
 (P_t \ot I_\initFock) V_\Omega
 \\
 &=
V_\Omega^*( s_\Omega \ot S )
 (\Sigma V_t \ot I_\initFock)
 (V_t^* \ot I_\initFock) V_\Omega
 \\
 &\subset
 V_\Omega^*
 ( \Sigma V_t \ot I_\initFock )
 ( K_t^\pi \ot S )
  (V_t^* \ot I_\initFock)
 V_\Omega
 \\
 &=
 V_\Omega^*
 ( \Sigma \ot I_\initFock )
 ( V_t \ot I_\initFock )
 ( V_t^* K^\pi V_t \ot S )
  (V_t^* \ot I_\initFock)
 V_\Omega
 \\
 &=
  V_\Omega^*
 ( \Sigma \ot I_\initFock )
 ( K^\pi \ot S )
 (V_t V_t^* \ot I_\initFock)
 V_\Omega
 \\
 &=
 V_\Omega^*
 ( \Sigma \ot I_\initFock )
 ( K^\pi \ot S )
  V_\Omega P_t^\Omega
 =
  \text{RHS of }\eqref{L}
\end{align*}
with equality if assumption ($\text{e}_+$) holds.
\end{proof}

\begin{lemma}
\label{fact}
 The operator $K^\pi \ot S$ on
 $L^2(\Rplus; \noisetwo \ot \initFock) = \Kiltwo \ot \initFock$
 may be characterised as follows\tu{:}
 \begin{align*}
 &\Dom  K^\pi \ot S =
 \big\{
 f \in L^2(\Rplus; \noisetwo \ot \initFock):
 f(t) \in \Dom k^\pi \ot S \text{ for a.a. } t, \text{ and }
 \\
 &\qquad \qquad \qquad \qquad \qquad \qquad \qquad \qquad
 \qquad \qquad
 (k^\pi \ot S)f(\cdot) \in L^2(\Rplus; \noisetwo \ot \initFock)
 \big\}
 \\
 &(K^\pi \ot S)f = (k^\pi \ot S) f(\cdot).
 \end{align*}
\end{lemma}
 \begin{proof}
 Call the operator defined above $R$.
 The inclusions
 $K^\pi \otul S \subset R \subset K^\pi \ot S$ are easily verified,
 it therefore suffices to show that $R$ is closed.
 Letting $(f_n)$ be a sequence in $\Kiltwo \otul \Dom S$ satisfying
 $f_n \to f$ and $R f_n \to g$, we may pass to a subsequence and
 assume that the convergence is almost everywhere.
 Then, for almost all $t\in \Rplus$,
 \[
 f(t) = \lim f_n(t) \text{ and }
 g(t) = \lim (Rf_n)(t) = \lim (k^\pi \ot S)f_n(t),
 \]
 and so, since $k^\pi \ot S$ is closed,
 $f(t) \in \Dom k^\pi \ot S$ and $(k^\pi \ot S) f(t) = g(t)$.
 Since $g$ is square-integrable, it follows that
 $f\in \Dom R$ and $R f = g$. Thus $R$ is closed, as required.
 \end{proof}

 Define the following modified It\^o integral:
 \begin{equation}
 \label{ItoSigma defined}
 \ItoSigma :=
 \Ito \circ (\Sigma \ot_\Omega I_\initFock),
 \end{equation}
 and set
 $\ItoSigma_t  :=
 \Ito \circ (\Sigma \ot_\Omega I_\initFock) P^\Omega_t$
 ($t\in\Rplus$).

 \begin{rem}
 Under Assumption ($\text{e}_+$), the integral
 $\ItoSigma_t$ is bounded and has full domain
 $L^2_\Omega(\Rplus; \noisetwo \ot \initFock)$,
 for all $t\in \Rplus$.
 Without Assumption ($\text{e}_+$)
 the domains may be smaller.
 Accordingly, let
 $\Domloc \Sigma \ot_\Omega I_\initFock$
 denote
 the set of (measure equivalence classes of)
 functions $z : \Rplus \to \noisetwo \ot \initFock$
 such that, for all $t \in \Rplus$,
 \begin{equation}
 \label{Domloc defined}
 z_t \in \noisetwo \ot \initFock_t
 \text{ and }
 z_{[0,t]} \in \Dom \Sigma \ot_\Omega I_\initFock.
 \end{equation}
 \end{rem}

 \begin{propn}
 \label{ItoSigma isometry}
 Let $z \in \Domloc \Sigma \ot_\Omega I_\initFock$.
 Then,
 for all $t\in \Rplus$,
 \[
 \norm{\ItoSigma_t z}^2
  =
  \norm{(\Sigma \ot I_\initFock)z_{[0,t]}}^2
 \]
 and $\ItoSigma_t z = 0$ if and only if $z_{[0,t]} = 0$.
 \end{propn}
 \begin{proof}
 The first part follows immediately from It\^o isometry.
 For the second part, note that we also have
 $
 \norm{\ItoSigma_t z}
 =
 \norm{\Sigma_t V_t^* z_{[0,t]}}$,
 and so the result follows from the injectivity of $\Sigma_t$
 and the fact that $V_t^*$ is isometric on $\initFock_t$.
 \end{proof}

 By a vector martingale in $\initFock$ we mean a family
 $(x_t)_{t\geq 0}$ in $\initFock$ satisfying $P_s x_t = x_s$
 for all $0\leq s \leq t$.

\begin{thm}
\label{X}
 Let $x$ be a vector martingale in $\initFock$.
 Then the following hold.
 \begin{alist}
 \item
  There is a unique
 $z\in \Domloc \Sigma \ot_\Omega I_\initFock$
 such that
 \[
 x_t - x_0 = \ItoSigma_t z
 \qquad
 (t\in\Rplus)
 \]
 \item
 The following are equivalent\tu{:}
 \begin{rlist}
 \item
 $x$ is $\Dom S$-valued.
 \item
 $x_0  \in \Dom S$ and
 $z_{[0,t]} \in \Dom  \Sigma K^\pi \ot_\Omega S$
 for all $t\in \Rplus$.
 \end{rlist}
 When these hold,
 \[
 \big(
 ( K^\pi \ot_\Omega S ) z_{[0,t]}
 \big)(s)
 =
 ( k^\pi \ot S ) z_s
 \qquad
 \text{ for a.a. } s \in [0,t],
 \]
 and, for all $t \in \Rplus$,
\begin{equation}
 \label{Sxt}
 S x_t - S x_0 = \ItoSigma_t \big( (k^\pi \ot S) z_\cdot \big) =
 \ItoSigma \big( ( K^\pi \ot_\Omega S ) z_{[0,t]} \big).
 \end{equation}
 \item If also \tu{(}$\text{e}_+$\tu{)} holds then we have the following further equivalences\tu{:}
 \begin{rlist}
 \item[(iii)]
 $x_0  \in \Dom S$ and
 $z_{[0,t]} \in \Dom  K^\pi \ot_\Omega S$
 for all $t\in \Rplus$.
 \item[(iv)]
 $x_0  \in \Dom S$,
 $z$ is almost everywhere
 $\Dom ( k^\pi \ot S )$-valued,
 and the function
 $s \mapsto ( k^\pi \ot S ) z_s$ is
 locally square-integrable.
 \end{rlist}
\end{alist}
\end{thm}
\begin{proof}
(a)
 Uniqueness follows from Proposition~\ref{ItoSigma isometry}.
 By the abstract Kunita-Watanabe Theorem (see~\cite{LQSI}),
 there is
 $y \in L^2_{\Omega, \loc}(\Rplus; \noisetwo\ot\initFock)$
 such that
 $x_t - x_0 = \Ito_t y$ ($t\in\Rplus$).
 Letting
 $z \in \Domloc \Sigma \ot_\Omega I_{\initFock}$
 be the process defined by
 \[
 z_{[0,t]} =
 \big( \Sigma^{-1} \ot_\Omega I_\initFock \big) y_{[0,t]},
 \]
 we have $x_t - x_0 = \ItoSigma_t z$ ($t\in\Rplus$).

 (b)
 By Lemma~\ref{tedious} we have
 \[
 ( s_\Omega \ot_\Omega S )
 (\Sigma \ot_\Omega I_\initFock)
 P_t^\Omega \subset
 ( \Sigma \ot_\Omega I_\initFock )
 ( K^\pi \ot_\Omega S \big)
 P_t^\Omega
 \qquad
 (t \in \Rplus).
 \]
 Therefore, by Part (c) of Theorem~\ref{S Omega Ito}
 (which happily ampliates to the current setting),
 \begin{align}
 S \Ito \circ
 (\Sigma \ot_\Omega I_\initFock) P_t^\Omega
 &=
 \Ito \circ
 (s_\Omega \ot_\Omega S)
 (\Sigma \ot_\Omega I_\initFock) P_t^\Omega
 \nonumber \\
 &\subset
 \Ito \circ ( \Sigma \ot_\Omega I_\initFock )
 ( K^\pi \ot_\Omega S )
 P_t^\Omega
 \label{JustAnotherItoFormula} \end{align}
 so $S \ItoSigma \circ  P_t^\Omega \subset
 \ItoSigma \circ
 ( K^\pi \ot_\Omega S )
 P_t^\Omega$,
  for all $t\in\Rplus$.
 This gives (i) $\Rightarrow$ (ii) and, when (i) holds, identity\eqref{Sxt}.
 Conversely, if (ii) holds then $(V_t^*\otimes I_\initFock)z_{[0,t]}$
 is in $\Dom \Sigma_tK^\pi_t\otimes S \cap \Dom \Sigma_t\otimes I$
 so, by Proposition \ref{cor to E-O} (c),
 $(V_t^*\otimes I_\initFock)z_{[0,t]}\in
 \Dom (s_{\Omega_{t)}}\otimes S)(\Sigma_t\otimes I_\initFock)$
 and it follows from Theorem \ref{Ito second quantisation commutation}
 that $z_{[0,t]}\in \Dom S\ItoSigma $, so (i) holds.

 (c) Now assume that \tu{(}$\text{e}_+$\tu{)} holds.
 Lemma~\ref{tedious} yields equality in (\ref{JustAnotherItoFormula}),
 so (i) is equivalent to (iii).
 The equivalence of (iii) and (iv) follows from
  Lemma~\ref{fact}.
\end{proof}

\section{Quasifree processes, martingales and integrals}
 \label{section: qf integrals}

 For this section the setup is the same as in
 Section~\ref{section: modified Ito},
 and we write $\Xi$ for the domain $\InitNoise' \xi$,
 as in Sections~\ref{section: affiliated operators}
 and~\ref{section: transpose and conjugate}.
 Quasifree martingales and stochastic integrals are defined
 and the martingale representation theorem is established.

 We rely heavily on the
 vector-operator
 linear isomorphisms~\eqref{eqn: opdagger dom S}
 and~\eqref{eqn: op vector},
 and on  the transpose operation on unbounded operators
 treated in Section~\ref{section: transpose and conjugate}.
 Filtrations of
 $\Op_\InitNoise(\Xi; \initFock)$ and
 $\Opdagger_\InitNoise(\Xi)$,
 and conditional expectations,
 are defined by
 \begin{align*}
 &\Op_\InitNoise(\Xi; \initFock)_t :=
 \big\{
 T \in \Op_\InitNoise(\Xi; \initFock): \,
 T \xi \in \initFock_t
 \big\},
 \\
 &
 \Opdagger_\InitNoise(\Xi)_t :=
 \Op_\InitNoise(\Xi; \initFock)_t \cap \Opdagger_\InitNoise(\Xi),
 \text{ and }
 \\
 &
 \mathbb{E}^\Sigma_t: \Op_\InitNoise(\Xi; \initFock) \to \Op_\InitNoise(\Xi; \initFock),
 \quad
 \mathbb{E}^\Sigma_t[T] \xi = P_t T \xi
 \qquad
 (t\in\Rplus).
  \end{align*}
Thus
 $\Op_\InitNoise(\Xi; \initFock)_t = \Ran \mathbb{E}^\Sigma_t$ and
 $\mathbb{E}^\Sigma_t\big[ \Opdagger_\InitNoise(\Xi) \big] =
 \Opdagger_\InitNoise(\Xi)_t$ ($t\in\Rplus$).
 A \emph{quasifree process} is a family
 $X = (X_t)_{t\geq 0}$ in $\Op_\InitNoise(\Xi; \initFock)$ adapted to
 the above filtration; it is a \emph{quasifree martingale} if
 it satisfies
 \[
 \mathbb{E}^\Sigma_s [X_t] = X_s
 \qquad
 (s\leq t),
 \]
 equivalently,
 $(X_t \xi)_{t\geq 0}$
 is a vector martingale
 with respect to the filtration
 $\big(\initFock_t\big)_{t\geq 0}$ (cf.~\cite{LQSI}).
 Thus, for example,
 if $T \in \Op_\InitNoise(\Xi; \initFock)$ then
 $\big(\mathbb{E}^\Sigma_t [T]\big)_{t\geq 0}$ is a
 martingale;
 these are called
 \emph{closed martingales}.

 \begin{rem}
 The maps $\mathbb{E}^\Sigma_t$ induce conditional expectations
 in the standard sense of Umegaki (norm-one projections)
 from $\InitNoise$
 to $\InitNoise_t := \Init \otol \Noise_t$
 which leave the vector state $\omega_\xi$ invariant.
 Here $\Noise_t := \Weyl_\Sigmao(\subspace_t)''$
 In general, due to Takesaki's No Go Theorem,
  the existence of such conditional expectations
 is not guaranteed;
 it rests on the subalgebras being left invariant
 by the modular automorphism group associated with
 $(\InitNoise, \xi)$
 (\cite{takesakiCE}; see Theorem IX.4.2 of~\cite{TakesakiTwo}).
 \end{rem}
 Write
 $\ProcSigma\kAu$ and $\MartSigma\kAu$ for the collection of
 quasifree processes, respectively martingales, and set
 \begin{align*}
 &\ProcSigmadagger \kAu :=
 \big\{
 X \in \ProcSigma\kAu:
 X_t \in \Opdagger_\InitNoise(\Xi)
 \text{ for all } t \in \Rplus
 \big\}, \text{ and }
 \\
 &\MartSigmadagger\kAu :=
 \MartSigma\kAu \cap \ProcSigmadagger \kAu,
 \end{align*}
 referring to such processes and martingales as \emph{adjointable}.
 We are ready to define quasifree stochastic integrals.
 Recall Corollary~\ref{Corollary 3.3}.

 \begin{defn}
 A \emph{quasifree integrand} is a
 family
 $F = \big(F_t)_{t\geq 0}$ in
 $\umatrix_{\khat}(\InitNoise, \xi)_0$
 such that
 \begin{equation}
 \label{QF integrand}
 \colFcdot \xi \in
 \Domloc \Sigma\ot_\Omega I_\initFock.
 \end{equation}
 Write $\IntegSigma\kAu$ for the collection of these,
 and $\IntegSigmadagger\kAu$ for the subcollection of
 \emph{adjointable integrands}, that is
 those for which
 \[
 F_t \in \umatrixdagger_{\khat}(\InitNoise, \xi)_0
 \text{ for all } t \in \Rplus
 \text{ and }
 F^\dagger := \big( F^\dagger_t\big)_{t\geq 0}
 \in
 \IntegSigma\kAu.
 \]
 For $F \in \IntegSigma\kAu$ define
 $\QSSigma_t(F) \in \Op_{\InitNoise}(\Xi; \initFock)$ by
 \[
 \QSSigma_t(F) \xi =
 \ItoSigma_t \big( \colFcdot \xi \big)
 \qquad (t\in\Rplus).
 \]
 \end{defn}
\begin{rems}
 (i)
 By Lemma~\ref{ultrastrong lemma},
 the operators of an adjointable quasifree process have
 common core $\Xi_0:= \Init' \cyclic \otul \Weyl_{\Sigmaoprime} \Omega$;
 those of an adjointable quasifree integrand have common core
 $\khat \otul \Xi_0$.

 (ii)
 The explicit action of quasifree integrals
 on vectors from the dense subspace
 $\Xi_0$
 is given by a
 Hitsuda-Skorohod integral  (\cite{LM});
 it is obtained from commution relations between
 Weyl operators and such integrals
 (\cite{LQSI}).

 (iii)
 For $F \in \IntegSigma\kAu$
 with block matrix form
 $\left[\begin{smallmatrix}
 0 & M \\ L & 0
 \end{smallmatrix}\right]$,
 $\colF$ is the family
 $\left[\begin{smallmatrix}
 L_\cdot \\ M_\cdot^\Transpose
 \end{smallmatrix}\right]$
 in $\ucol_{\noisetwo}(\Mil,\xi) =
 \Op_\Mil(\Xi; \noisetwo\ot\initFock)$,
 and if
 $F \in \IntegSigmadagger\kAu$ then
 $\colFt \in
 \ucoldagger_{\noisetwo}(\Mil, \xi) \subset
 \Opdagger_\Mil(\Xi, \noisetwo \otul \Xi)$
 and
 $\colFdagger =
 \left[\begin{smallmatrix}
 M_\cdot^\dagger \\ L_\cdot^{\Transpose \dagger}
 \end{smallmatrix}\right]$.

 (iv)
 The top left zero in the block matrix form of
 $F$ is available for a time-integral.
 We have no need for these here,
 but they arise naturally in~\cite{LM}.

 (v)
 The bottom right zero is related to the fact that there is no
 number/exchange/guage process affiliated to the quasifree
 filtration.

  (vi)
 From Proposition~\ref{ItoSigma isometry} we have
 a form of \emph{It\^o isometry} (cf.~\cite{BSWtwo}):
 \[
 \norm{\QSSigma_t(F) \xi}^2 =
 \norm{\Sigma_t V_t^* z_{[0,t]}}^2,
 \quad
 \text{for all }t \in \Rplus,~\text{where } z := \colFcdot \xi.
 \]

 (vii)
 Quasifree creation and annihilation integrals are defined by
 \[
 A^*_t(L) +  A_t(M) =
  \QSSigma_t (F)
 \text{ where }
 F = \left[\begin{smallmatrix} & M \\ L & \end{smallmatrix}\right].
 \]
 The proposition below confirms that, for adjointable $L$,
 $A^*_t(L)^\dagger = A_t(L^\dagger)$.

 (viii)
 When $\subspace$ is a complex subspace of $\Kil$,
 as in the case of squeezed states,
 quasifree creation and annihilation \emph{operators}
 (may be formed, and)
 may be
 viewed as quasifree \emph{Wiener integrals}:
 \begin{equation*}
 \big( a^*(f) + a(g) \big) =
 \QSSigma ( H ),
 \ \text{ so } \
 \big( a^*(f) + a(g) \big) \xi =
 \ItoSigma \big( h \ot \xi \big)
  \qquad
 (f,g\in \Kil),
 \end{equation*}
 where
 \begin{equation*}
 H = \begin{bmatrix}
 & \bra{g} \ot I_\initFock \\
 \ket{f} \ot I_\initFock &
 \end{bmatrix}
 \text{ and }
 h = \binom{f}{\ol{g}}.
 \end{equation*}

 (ix)
 In the guage-invariant case we have orthogonality of
 creation and annihilation integrals
 on the cyclic and separating vector,
 entailing some simplification in the analysis for that case:
 \[
 A^*_t(L) \xi \perp A_t(M) \xi
 \qquad (t\in\Rplus).
 \]

 (x)
 Under ($\text{e}_+$), the condition of adjointability for
 $F \in \IntegSigma\kAu$
 is equivalent to
 \[
  (F^{[]}_\cdot\xi)_{[0,t]}\in\Dom K\otimes_\Omega S\quad \text{for all }t\in\Rplus,
 \]
 which is in turn equivalent to
 \begin{align*}
  & F_t\in \umatrixdagger_{\khat}(\InitNoise, \xi)_0\quad\text{for a.a. }t\in\Rplus,\text{ and} \\
  & (k\otimes S)F^{[]}_\cdot \xi \quad \text{is locally square integrable}.
 \end{align*}
 \end{rems}

 \begin{example}[Exponential martingales]
 Elementary examples of \emph{bounded} quasifree martingales
 are given by
 \[
 E^f_t =
 e^{\frac{1}{2} \norm{\Sigma \iota (f_{[0,t]})}^2} W(f_{[0,t]})
 \qquad
 (t\in\Rplus),
 \]
 where $f \in L^2_{\loc}(\Rplus; \noise)$ is such that
 $\iota(f) \in \Domloc \Sigma$ (so $\Init = \Comp$ here).
 These martingales are adjointable, with $(E^f)^\dagger = E^{-f}$, and
 have the following
 stochastic integral representation:
 \[
 E^f = I_\Fock + \QSSigma_\cdot(F)
 \ \text{ where } \
 F_t = i
  \left[\begin{smallmatrix}
   & \bra{f(t)} \\ \ket{f(t)} &
   \end{smallmatrix}\right]
 \ot E^f_t
 \qquad
 (t\in\Rplus).
 \]
 In other words, they satisfy the basic quasifree stochastic differential
 equation
 \[
 dE^f_t = E^f_t d X^f_t
 \quad
 E^f_0 = I_\Fock,
 \]
 where
 $X^f$ is the martingale
 formed from the field operators
 $\big( i R(f_{[0,t]}) \big)_{t\in\Rplus}$;
 $E^f$ is said to be the \emph{stochastic exponential} of $X^f$.
\end{example}

 \begin{propn}
 \label{integrals are martingales}
 Let $F \in \IntegSigma\kAu$.
 Then
 $\QSSigma_\cdot(F) \in \MartSigma\kAu$.
 \end{propn}
 \begin{proof}
 This follows from the fact that,
 for any $z \in \Domloc(\Sigma\ot_\Omega I_\initFock)$,
 $\ItoSigma_\cdot(z)$ is
 (an It\^o-integral process and thus)
 a vector martingale.
\end{proof}

 We conclude with the converse, which may be viewed as
 confirmation that the general definition of quasifree integrals
 given here is the correct one.

\begin{thm}
\label{main theorem}
 Let $X \in \MartSigma\kAu$.
 Then the following hold.
 \begin{alist}
 \item
 There is a unique $F \in \IntegSigma\kAu$
 such that
 \begin{equation}
 \label{star}
 X_t - X_0 = \QSSigma_t(F)
 \qquad
 \qquad (t \geq 0).
 \end{equation}
 \item
 The martingale
 $X$ is adjointable if and only if
 the operator $X_0$ is adjointable and
 the integrand process
 $F$ is adjointable. In this case
 \[
 X^\dagger_t - X^\dagger_0 = \QSSigma_t(F^\dagger)
 \qquad
 \qquad (t \geq 0).
 \]

 \end{alist}
\end{thm}
\begin{proof}
 (a)
 Uniqueness follows from uniqueness in Theorem~\ref{X}.
 Let
 $x = \big(X_t \xi\big)_{t\geq 0}$
 be the corresponding vector process in $\initFock$.
 Then,
 by Theorem~\ref{X},
 there is a unique
 $z\in \Domloc \Sigma \ot_\Omega I_\initFock$
 such that
 $x_t - x_0 = \ItoSigma_t z$ for all $t\in \Rplus$.
 Now define
 \[
 Q_t := \ket{z_t}^{\xi} \in
 \Op_\InitNoise(\Xi; \noisetwo \ot \initFock) =
 \ucol_{\noisetwo}(\InitNoise, \xi)
 \]
 and,
 recalling Corollary~\ref{Corollary 3.3},
 define $F_t \in \umatrix_{\khat}(\Mil,\xi)_0$ by
 $\colFt = Q_t$ ($t\in\Rplus$). Then
 $
  Q_\cdot \xi = z \in
 \Domloc \Sigma \ot_\Omega I_\initFock
 $
 and so
 $F \in \IntegSigma\kAu$ and~\eqref{star} holds since
 \[
 \QSSigma_t(F) \xi =
 \ItoSigma_t (Q_\cdot \xi) =
 x_t - x_0 =
 (X_t - X_0) \xi.
\]

 (b)
 Now suppose that the operator $X_0$ is adjointable.
 By Theorem~\ref{Theorem 3.2},
 the adjointability of the integrand process $F$
 is equivalent to
 \begin{align*}
 &Q_\cdot \xi \text{ is a.e. }
 \Dom (k \ot S)\text{-valued, and  }
 \\
 &(k \ot S) Q_\cdot \xi \in \Domloc \Sigma\ot_\Omega I_\initFock.
 \end{align*}
 Since $\pi$ is unitary,
 $k$ may be replaced by $k^\pi = \pi \circ k$ and so,
 by Theorem~\ref{X},
 this is equivalent to
 \[
 (x_t - x_0) \in \Dom S \text{ for all } t \in \Rplus,
 \]
 in which case,
 \[
 S x_t - S x_0 =
 \ItoSigma_t \big(
 ( k^\pi \ot S) z_\cdot
 \big)
 \qquad
 \text{ for all } t \in \Rplus.
 \]
 Thus $F$ is adjointable if and only if $X$ is adjointable,
 in which case,
 by Corollary~\ref{Corollary 3.3},
\[
 X^\dagger_t \xi - X^\dagger_0 \xi
 =
  \ItoSigma_t\big( ( k^\pi \ot S ) Q_\cdot \xi \big)
 =
 \ItoSigma_t \big( \colFdaggercdot \xi \big)
 = \QSSigma_t(F^\dagger) \xi.
 \]
 (b)
 follows and so the proof is complete.
\end{proof}

\begin{rem}
 If Assumption \tu{($\text{e}_+$)} also holds then, by Remark (x)
 following the definition of quasifree integrands,
 the conditions for $F$ to be adjointable simplify.
\end{rem}

\renewcommand{\theequation}{\Alph{section}.\arabic{equation}}

\renewcommand{\thesection}{\Alph{section}}

\setcounter{section}{0}

\renewcommand{\thepropn}{\Alph{section}.\arabic{propn}}

\section*{Appendix: Unbounded operators and tensor products}
 \label{section: unbounded operators}

\setcounter{section}{1} \setcounter{propn}{0}

 In this appendix we collect some basic facts about the behaviour of
 unbounded linear and conjugate-linear operators under composition,
 adjoint, orthogonal sum and tensor operations,
 for ease of reference in the paper.

 For compatible densely defined Hilbert space operators we have
 the following inclusions
 \begin{align}
 &(S_1 + \lambda S_2)^* \supset S_1^* + \ol{\lambda} S_2^*,
 \ \text{ with equality if } S_1 \text{ is bounded,}
 \nonumber
  \\
 &(S_3S_4)^* \supset S_4^*S_3^*,
 \ \text{ with equality if } S_3 \text{ is bounded},
  \label{adjoint product}
 \end{align}
 whenever
 $S_1 + \lambda S_2$ and $S_3 S_4$ are also densely defined and
 $\lambda \in \Comp\setminus \{0\}$.
 We refer to~\eqref{adjoint product} as the
  \emph{adjoint-product-inclusion relation}.
  We call a Hilbert space operator $T$, with target $\Hil$,
  injective/surjective/bijective if it has that property
  as a map from $\Dom T$ to $\Hil$. Thus if $T$ is injective
  then $T^{-1}$ is the operator given by $\Dom T^{-1} = \Ran T$,
  $Tu \mapsto u$; if $T$ is closed and bijective then $T^{-1}$ is
  everywhere defined and, by the Closed Graph Theorem, bounded ---
  as is usual, we refer to such operators as invertible.
 Here are some more detailed relations.
 They each follow, in turn, from the definitions;
 proofs of (a) and (b) may be found, for example, in~\cite{Weidmann}.
 Recall that a core for an operator $T$ is a subspace of its domain
 which is dense in the graph norm of $T$.

\begin{lemma} \label{Lemma 1}
 Compatible Hilbert space operators satisfy the following.
\begin{alist}
 \item
 Let
  $S, B, R, E$ and $F$ be operators, with
 $S$ closable,
 $B$ bounded,
 $R$ closed and injective with bounded inverse,
 $E$ bounded, everywhere defined and bijective, and
 $F$ bounded and injective with bounded inverse.
 Then \tu{(}when defined\tu{)}
 \begin{rlist}
 \item
 $\ol{S} B$ and $R \ol{S}$ are closed\tu{;}
 \item
 if $B S$ is closable and $\Dom B \supset \Ran S$ then
 $B \ol{S}$ is closable and
 \[
 \ol{B \ol{S}} = \ol{B S}\tu{;}
 \]
 \item
 $FSE$ is closable and
 \[
 \ol{FSE} = F \ol{S} E,
 \]
 in particular,
 $F \ol{S} E$ is closed with core $E^{-1} \Dom S$.
\end{rlist}
 \item
 Let $T$ be a closed and densely defined operator,
 and let $D$ be a closed, densely defined and bijective operator.
 Then \tu{(}when defined\tu{)}
 \[
 (T D)^* = D^* T^*.
 \]
 \item
 Let $S$ be a closable operator and
 $V$  an \tu{(}everywhere defined\tu{)} isometric operator
 satisfying
 $
 \ol{S} V V^* \supset V V^* \ol{S}.
 $
 Then $V^* \ol{S} V$ is closed and $V^*( \Dom S)$ is a core
 for both $\ol{S} V$ and $V^* \ol{S} V$.
 Moreover, if $S$ is also densely defined then
 \[
 (V^* \ol{S} V)^* = V^* S^* V.
 \]
\end{alist}
 \end{lemma}

 We need to consider tensor products of unbounded operators.
 The following commonly used notation is convenient.
 For operators $T_1$ and $T_2$,
 $T_1 \otul T_2$ denotes the unique operator $T$ satisfying
 \begin{align*}
 &\Dom T :=
 \Dom T_1\otul \Dom T_2
 \\
 &T (u_1 \ot u_2) =
 T_1 u_1 \ot T_2 u_2
 \quad
 (u_1 \in \Dom T_1, u_2 \in \Dom T_2).
 \end{align*}
  The elegant proof of part (c) below is from~\cite{Weidmann},
  it perhaps deserves to be better known; for other proofs, see
 Section VII.10 of~\cite{ReS} and Chapter 9 of~\cite{StZ}.
 Recall that, for an operator $T$ on $\Hil$, a vector $x\in\Hil$ is
 \emph{analytic for $T$} if $x\in\bigcap_{n\in\Nat} \Dom T^n$
 and $\sum_{n\geq 0} (n!)^{-1} \norm{(tT)^n x} < \infty$, for some
 $t > 0$.

 \begin{lemma}
 \label{Lemma 2}
 Let $T = T_1 \otul T_2$ for Hilbert space operators $T_1$ and $T_2$.
 \begin{alist}
 \item
 If $T_1$ and $T_2$ are closable then $T$ is too.
 \item
  If $T_1$ and $T_2$ are closable and densely defined then
  \begin{rlist}
  \item
   $T^* = \ol{T_1^* \otul T_2^*}$,
  \item
  $\ol{T} = (T_1^* \otul T_2^*)^*$.
  \end{rlist}
\item
 If $T_1$ and $T_2$ are essentially selfadjoint then $T$ is too.
 \end{alist}
 \end{lemma}
\begin{proof}
 (c)
 First note  that, being densely defined and symmetric,
 $T$ is closable, $\ol{T}$ is symmetric and
 $\ol{T} \supset \ol{T_1} \otul \ol{T_2}$.
 Let $\analytic_1$, $\analytic_2$ and $\analytic$
 denote respectively the space of analytic vectors for
 the operators
 $\ol{T_1}^2$, $\ol{T_2}^2$ and $\ol{T}$.
 It is easily verified that
 $\analytic \supset \analytic_1 \otul \analytic_2$.
 Since a closed symmetric operator is selfadjoint
 if and only if
 its space of analytic vectors is dense (\cite{nelson}; see Theorem
 X.39 of~\cite{RStwo}), (c) follows.

 (b)
 (ii) follows from (i) by taking adjoints. We prove (i).
 It is easily seen that $T^* \supset T_1^* \otul T_2^*$, so $T$ is
 closable, and that $\ol{T_1} \otul \ol{T_2} \subset \ol{T}$.
 We must show that $\Dom T^*_1 \otul T_2^*$ is a core for $T^*$.
 Suppose therefore that $z\in\Dom T^*$ is orthogonal to
 $\Dom T_1^* \otul \Dom T_2^*$ with respect to the graph inner
 product of $T^*$; we must show that $z=0$.
 Setting
 $A := \ol{T_1} T_1^* \otul \ol{T_2} T_2^*$, we have
 $A \subset \ol{T}T^*$ and, for all $u\in \Dom A$,
 \[
 0 =
  \ip{z}{u} + \ip{T^*z}{T^*u} =
  \ip{z}{(I+A)u}.
 \]
 By (c) $A$ is essentially selfadjoint and so
 $\ol{T}T^* = \ol{A}$.
 Now $I+\ol{T}T^*$ is invertible,
 so $I+A$ has dense range and thus $z=0$, as required.

 (a)
 This follows by applying (b) to the operators obtained by viewing
 $T_1$, $T_2$ and $T$ as densely defined operators from the Hilbert
 spaces $\ol{\Dom} T_1$, $\ol{\Dom} T_2$ and $\ol{\Dom} T$
 respectively.
\end{proof}

\begin{notn}
 For closed operators $R_1$ and $R_2$
 (following common practice) we set
 \[
 R_1 \ot R_2 := \ol{R_1 \otul R_2}.
 \]
 Thus, for closable densely defined operators $T_1$ and $T_2$,
 we have
 \begin{equation}
 \label{operator tensor}
 (T_1 \otul T_2)^* = T_1^* \ot T_2^* = (\ol{T_1} \ot \ol{T_2})^*.
 \end{equation}
\end{notn}
 The useful facts collected together next may all be proved
 by systematic application of the above two lemmas.

\begin{propn}
 \label{Lemma 3}
 For $i=1,2$, let
 $R_i$, $\wt{R}_1$, $T_i$, $B_i$, $\wt{B_1}$, $E_i$ and $F_i$
 be Hilbert space operators,
 with $R_i$ and $\wt{R}_1$ closed, $T_i$ closed and densely defined,
 $B_i$ and $\wt{B}_1$ bounded and everywhere defined,
 $E_i$ bounded, everywhere defined and bijective, and
 $F_i$ bounded, and injective with bounded inverse,
 and set
 \[
 R = R_1 \ot R_2, \
 T = T_1 \ot T_2, \
 B = B_1 \ot B_2, \
 E = E_1 \ot E_2, \
 F = F_1 \ot F_2, \
 \]
 and $\wt{R} = \wt{R}_1 \ot R_2$.
 Then the following hold
 \tu{(}when the compositions are defined\tu{):}
  \begin{alist}
 \item
 $RB \supset R_1 B_1 \ot R_2 B_2$.
 \item
 $TB = T_1 B_1 \ot T_2 B_2$ if $T_1B_1$ and $T_2B_2$ are densely defined.
 \item
 $RE = R_1 E_1 \ot R_2 E_2$.
 \item
 If $BR$, $B_1R_1$ and $B_2R_2$ are closable then
 $\ol{BR} = \ol{B_1R_1} \ot \ol{B_2R_2}$,
 in particular, $F R = F_1 R_1 \ot F_2R_2$.
 \item
 $T = \ol{(T_1\ot I_2)(I_1 \ot T_2)}$, and if either
 $T_1$ is injective with bounded inverse, or
 $T_2$ is bounded, then
 $(T_1\ot I_2)(I_1 \ot T_2)$ is closed, so
 $
 T = (T_1\ot I_2)(I_1 \ot T_2).
 $
 \item
 If $R_1 B_1 \supset \wt{B}_1 \wt{R}_1$ then
 $R (B_1 \ot I_2) \supset (\wt{B}_1 \ot I_2') \wt{R}$.
 \end{alist}
 \end{propn}

 The following corollary is also useful.

\begin{cor}
\label{tensor invariance}
 Let $T = T_1 \ot T_2$ and $U = U_1 \ot U_2$
 where, for $i=1,2$,
 $T_i$ is a closed and densely defined operator
 from $\Hil_i$ to $\Hil'_i $,
 $U_i$ is a closed subspace of $\Hil_i$,
 and $T_i \big( U_i \cap \Dom T_i \big) \subset U_i$.
 Then $T \big( U \cap \Dom T \big) \subset U$.
\end{cor}
\begin{proof}
 Letting
 $V_1$, $V_2$ and $V$ be the inclusion maps
 of $U_1$, $U_2$ and $U$
 in $\Hil_1$, $\Hil_2$ and $\Hil$ respectively,
 Part (b) of Proposition~\ref{Lemma 3} implies that
 \[
 T_{|U} = T V =
 T_1 V_1 \ot T_2 V_2 =
 T_{1\,|U_1} \ot T_{2\,|U_2},
 \]
 from which the result is evident.
\end{proof}

 For a sequence of operators
 \big($T_n$ from $\Hil_n$ to $\Hil'_n$\big)$_{n\geq 0}$
 an operator $T = \bigoplus T_n$ from
 $\Hil = \bigoplus  \Hil_n$ to
 $\Hil' = \bigoplus \Hil'_n$
 is defined in the obvious way:
 \[
 \Dom T =
 \Big\{
 \xi \in \Hil: \forall_{n\geq 0}\, \xi_n \in \Dom T_n \text{ and }
 \sum_{n\geq 0} \norm{T_n \xi_n}^2 < \infty
 \Big\},
 \quad
  T \xi = (T_n\xi_n).
 \]

 Elementary properties of this construction include the following:
 \begin{itemize}
  \item[*]
 If each $T_n$ is closed then so is $T$.
 \item[*]
 If each $T_n$ is densely defined then so is $T$, and
 $T^* = \bigoplus T_n^*$,
 \item[*]
 If each $T_n$ has core $\core_n$ then $T$ has core
 $\sum_{n\geq 0}^\oplus \core_n$ \tu{(}algebraic sum\tu{)}.
 \item[*]
 $T$ is bounded if and only if each $T_n$ is bounded and
 $\sup_n \norm{T_n} < \infty$.
 \end{itemize}
 Recall the notation
 $\Exps(S):= \Lin \{ \ve(v): v\in S\}$.
 For a closed operator $R$ from $\hil_1$ to $\hil_2$,
 operators  from $\Gamma(\hil_1)$ to $\Gamma(\hil_2)$
 are defined by
 \begin{align*}
 &\Gamma(R) := \bigoplus R^{(n)},
 \ \text{ where, for } n \geq 0, \
 R^{(n)} := {V'_n}^* R^{\ot n} V_n, \text{ and }
 \\
 & \Gamma(R)_| := \Gamma(R)_{|\Exps(\Dom R)},
 \end{align*}
 $V_n$ and $V'_n$ being the inclusions
 $\hil_1^{\vee n} \to \hil_1^{\ot n}$ and
 $\hil_2^{\vee n} \to \hil_2^{\ot n}$.

 \begin{propn}
 \label{second quantisation}
 Let $R$, $S$ and $T$ be operators from $\hil_1$ to $\hil_2$
 such that $S$ is densely defined,
 $R$ is closed and $T$ is closed and densely defined,
 and let
 $C$ be an everywhere defined contraction operator from $\hil_0$ to
 $\hil_1$.
 Then the following hold.
\begin{rlist}
 \item
 $\Gamma(R)$ is closed.
 \item
 If $\core$ is a core for $R$ then
 $\Exps(\core)$ is a core for $\Gamma(R)$.
 \item
 $\Gamma(S)^* = \Gamma(S^*)$.
 \item
 $\Gamma(C)$ is an everywhere defined contraction operator.
 \item
 $
 \Gamma(RC) \subset \Gamma(R) \Gamma(C)$.
 \item
 $\Gamma(TC) = \Gamma(T) \Gamma(C)
 $,
 when $TC$ is densely defined.
\end{rlist}
 \end{propn}
\begin{proof}
 (i), (iii) and (iv) follow easily from
 the elementary properties of orthogonal sums of operators listed above.
 (ii)
 follows from the fact that
 $\Exps(\core)$ is dense in
 $\Gamma(\hil^+)$, where
 $\hil^+$ denotes $\Dom R$ in the graph norm of $R$,
 and this in turn implies (v),
 in view of the obvious inclusion
 \[
 \Gamma(RC)_| \subset \Gamma(R) \Gamma(C),
 \]
 and the closedness of the RHS
 (by Part (a) of Lemma~\ref{Lemma 1}).

 (vi) follows from (v) and the fact that
 $T^{\ot n} C^{\ot n} = (TC)^{\ot n}$ ($n\in\Nat$),
 cf. Part (b) of Proposition~\ref{Lemma 3}.
\end{proof}

\begin{rem}
 For an everywhere-defined contraction operator $C$,
 $\Gamma(C)$ is
 known as the \emph{second quantisation} of $C$
 (\cite{cook} see, for example, \cite{ReS}).
\end{rem}

 We also need to consider conjugate-linear operators,
 including the Tomita-Takesaki operators associated with
 a von Nemann algebra with cyclic and separating vector.
 Thus, for a conjugate linear operator $T$ from $\Hil_1$ to $\Hil_2$
 with domain $\domain$, its \emph{adjoint} is
 the conjugate-linear operator from $\Hil_2$ to $\Hil_1$
 defined as follows:
 \begin{align*}
 &\Dom T^* :=
 \big\{
 x \in \hil': \text{ the linear functional } u \in \domain \to
 \ip{Tu}{x} \text{ is bounded}
 \big\}
 \\
 &\ip{T^*x}{u} = \ip{Tu}{x}
 \quad
 (u\in\domain, x \in \Dom T^*).
 \end{align*}
 In terms of any antiunitary operator $J: \Hil_2 \to \Hil_1$,
 \[
 T^* = J (TJ)^*,
 \]
 $TJ$ being a \emph{linear} operator with domain $J^{-1} \domain$.
 Compositions, orthogonal sums and tensor products of conjugate-linear
 operators enjoy
 corresponding properties to those of their linear sisters listed
 above.
 Thus, for closable conjugate-linear operators $T$, $T_1$ and $T_2$,
 $T_1 \otul T_2$ is closable and its closure is denoted $T_1 \ot
 T_2$, and $\Gamma(T)$ enjoys the properties listed in
 Proposition~\ref{second quantisation}.

 \begin{caution}
 If $T_1$ is a linear operator and $T_2$ a conjugate-linear operator
 then
 (except in the trivial case where one is a zero operator)
 $T_1 \otul T_2$  makes no sense, let alone $T_1 \ot T_2$.
 \end{caution}

\noindent \emph{ACKNOWLEDGEMENTS.}
 This work was supported by the UKIERI Research Collaboration Network grant
 \emph{Quantum Probability, Noncommutative Geometry \& Quantum Information}.

\end{document}